\documentclass{amsart}
\usepackage{graphicx,amsmath,amssymb,amsfonts,amsthm}
\usepackage[all,cmtip]{xy}
\usepackage[colorlinks=false,hyperindex]{hyperref}
\usepackage{color}
\usepackage{fullpage}
\usepackage{tikz}
\usepackage{tikz-cd}

\newtheorem{theorem}{Theorem}[section]
\newtheorem{corollary}[theorem]{Corollary}
\newtheorem{lemma}[theorem]{Lemma}
\newtheorem{proposition}[theorem]{Proposition}
\newtheorem{conjecture}[theorem]{Conjecture}

\theoremstyle{definition}
\newtheorem{definition}[theorem]{Definition}
\newtheorem{example}[theorem]{Example}
\theoremstyle{remark}
\newtheorem{remark}[theorem]{Remark}
\numberwithin{equation}{section}

\newcommand{\ZZ}{\mathbb{Z}}

\newcommand{\RR}{\mathbb{R}}

\newcommand{\Krull}{\operatorname{Krull}}
\newcommand{\trdeg}{\operatorname{trdeg}}
\newcommand{\ord}{\operatorname{ord}}

\renewcommand{\AA}{\mathbb{A}}

\newcommand{\Spec}{\operatorname{Spec}}

\newcommand{\Ocal}{\mathcal{O}}

\newcommand{\supp}{\operatorname{supp}}
\newcommand{\Irr}{\operatorname{Irr}}

\newcommand{\IN}{\operatorname{in}}

\newcommand{\sat}{\operatorname{sat}}

\newcommand{\red}{\operatorname{red}}

\newcommand{\height}{\operatorname{ht}}

\newcommand{\tdet}{\operatorname{tdet}}
\newcommand{\lex}{\operatorname{lex}}

\newcommand{\BS}{\operatorname{BS}}
\newcommand{\out}{\operatorname{out}}
\newcommand{\fix}{\operatorname{fix}}
\newcommand{\move}{\operatorname{move}}
\newcommand{\rem}{\operatorname{rem}}

\newcommand{\defi}[1]{\textsf{#1}} 				
\newcommand{\taylor}[1]{{\color{blue} \sf $\clubsuit\clubsuit\clubsuit$ Taylor: [#1]}}

\title{The Dimension Conjecture Implies The Jacobi Bound Conjecture}
\author{Taylor Dupuy, David Zureick-Brown}
\date{\today}

\begin{document}

\maketitle
\begin{abstract}
	We prove that the Dimension Conjecture implies the Jacobi Bound Conjecture.
\end{abstract}
\tableofcontents

\section{Introduction}

This paper shows that the two central problems in the dimension theory of algebraic differential equations, the Dimension Conjecture and the Jacobi Bound Conjecture, are equivalent.
This paper supplies the direction that the Dimension Conjecture implies the Jacobi Bound Conjecture.
 
Both of these conjectures concern the behavior of solutions of nonlinear systems of ordinary algebraic differential equations in several dependent variables $x_1(t),x_2(t),\ldots,x_n(t)$, with coefficients in a field of meromorphic functions in a single complex variable $t$. These conjectures are based on intuitions about bounding the number and type of parameters needed for the ``general solution'' of an irreducible component of such a nonlinear system.
One type of parameter a solution can have is a dependence on an arbitrary function of a complex variable and another type of dependence a solution can have is on a constant of integration.

While we are speaking analytically for a moment to the general audience, we want to make clear that these conjectures answer these questions algebraically, not analytically, and the notions of \defi{differential dimension} $\dim^{\partial}(\Sigma)$ and \defi{dimension} $\dim(\Sigma)$ associated to a general nonlinear system $\Sigma$ are the ``correct'' definitions for what are classically called ``the number of arbitrary functions of a complex variable on which a solution depends'' and ``the number of constants of integration on which a solution depends''.
The invariant $\dim(\Sigma)$ is the algebraic replacement for ``the number of constants of integration'' while the invariant $\dim^{\partial}(\Sigma)$ is the algebraic replacement for ``the number of arbitrary functions of a complex variable''.

The Dimension Conjecture appears as Ritt's 10th Problem in his list of open problems in differential algebra \cite[Appendix]{Ritt1950}. 
Intuitively, it states that if a system of ordinary differential equations has fewer equations than dependent variables, then the a general solution of any component must depend on at least one free meromorphic function of the complex variable $t$. 
In the rigorous algebraic formulation this translates into a statement about differential dimensions.

The Jacobi Bound Conjecture appears as Ritt's 11th Problem in \cite{Ritt1950} and as Kolchin's Problem 3B in Kolchin's \emph{Some Problems in Differential Algebra} which accompanied his 1968 ICM address \cite{Kolchin1968}.\footnote{There are three italicized problems appearing in section 3 of \cite{Kolchin1968}.
	We are referring to these as 3A, 3B, and 3C.
} 
The conjecture is based on a statement found in a manuscript of Jacobi from 1865 \cite{Jacobi1865,Jacobi2009} and deals with the situation where the number of dependent variables and the number of differential equations are equal: if 
$$u_1=0, \quad u_2=0, \quad \ldots, \quad u_n=0,$$ 
is a system of nonlinear differential equations in $n$ dependent functions of a single complex variable $t$, call them $x_1(t),x_2(t),\ldots,x_n(t)$, then number of constants of integration required for a general solution $(x_1(t),x_2(t),\ldots,x_n(t))$ of an irreducible component of finite dimension is bounded by an explicit constant called the \defi{Jacobi bound} $J(u_1,\ldots,u_n)$ given by
\begin{equation}\label{E:jacobi-bound}
	J(u_1,\ldots,u_n) = \max_{\rho \in S_n} \sum_{i=1}^n \ord_{x_i}^{\partial}(u_{\rho(i)}).
\end{equation}
Here $S_n$ is the symmetric group of permutations of $\lbrace 1,2,\ldots,n\rbrace$ and $\ord_{x_i}^{\partial}(u_j)$ is the largest derivative of $x_i$ appearing in equation $u_j$.

Later, Ritt in 1935 \cite{Ritt1935} realized that Jacobi's statement had no rigorous definition of ``number of constants of integration''\footnote{Jacobi and Ritt themselves used the word ``order'' for this undefined notion of dimension. Also Jacobi had no notion of components and never used the words ``Jacobi bound''. 
A close account of what Jacobi did and didn't do in modern language can be found in Ollivier's wonderful paper \cite{Ollivier2022}.
} and provided a definition based on transcendence degrees. 
This is the invariant $\dim(\Sigma)$ discussed previously.
In his paper \cite{Ritt1935}, Ritt declared this an open problem and proved several cases about the conjecture.

As is common practice in differential algebra, we now pass to general differential fields and work algebraically to give our formal statements. Let $(K,\partial)$ be a differential field of characteristic zero. 
Denote by $K\lbrace x_1,\ldots,x_n\rbrace$ the ring of differential polynomials over $K$. 

\begin{conjecture}[Jacobi Bound Conjecture {\cite{Kolchin1968}}]\label{C:jbc-intro}
	Let $K$ be a differential field of characteristic zero. 
	Let $u_1,u_2,\ldots,u_n\in K\lbrace x_1,\ldots,x_n\rbrace$.
	Let $\Sigma=\Spec K\lbrace x_1,\ldots,x_n \rbrace/[u_1,\ldots,u_n].$
	Let $\Sigma_1$ be an irreducible component of $\Sigma$.
	Suppose that $\dim_K^{\partial}(\Sigma_1)=0$.
  Then $\dim_K(\Sigma) \leq J(u_1,\ldots,u_n)$.
\end{conjecture}
The Jacobi Bound $J(u_1,\ldots,u_n)$ given in equation \eqref{E:jacobi-bound} and which appears in Conjecture~\ref{C:jbc-intro} comes in two forms: a \defi{weak bound} and \defi{strong bound}. 
Accordingly, Jacobi Bound Conjecture is called \defi{weak} or \defi{strong} according to whether we are using the weak or strong conventions for $\ord_{x_i}^{\partial}$ (these give the weak and strong bounds respectively); the \defi{weak convention} sets $\ord^{\partial}_y(f)=0$ if neither $y$ nor any of its derivatives appear in the differential polynomial $f$ while the \defi{strong convention} sets $\ord_y^{\partial}(f)=-\infty$ in such a situation (see Section~\ref{S:notation}).

Both the weak and strong Jacobi Bound Conjectures are open and the strong Jacobi Bound Conjecture implies the weak Jacobi Bound Conjecture simply because the strong Jacobi bound is less than or equal to the weak Jacobi bound.

\begin{conjecture}[Dimension Conjecture {\cite{Cohn1983}} {\cite[pg 178]{Ritt1950}}\footnote{Cohn doesn't mention Ritt's formulation and attributes it to Lando so it may be an independent formulation of this conjecture.}]\label{C:original}
	Let $K$ be a differential field of characteristic zero. 
	Let $u_1,\ldots,u_m \in K\lbrace x_1,\ldots,x_n\rbrace$. 
	Let 
	$$ \Sigma = \Spec K\lbrace x_1,\ldots,x_n\rbrace/[u_1,\ldots,u_m].$$
	If $m<n$, then for every irreducible component $\Sigma_1$ of $\Sigma$ we have $\dim_K^{\partial}(\Sigma_1)\geq n-m$.
\end{conjecture}

The strong version of the Jacobi Bound Conjecture is known to imply the Dimension Conjecture \cite{Cohn1983}.\footnote{
Suppose strong JBC and consider $[u_1,\ldots,u_m] \subset K\lbrace x_1,\ldots,x_n\rbrace$ with $m<n$. 
Let $P$ be a minimal prime above $[u_1,\ldots,u_m]$. 
Extend the system by declaring $u_{m+1}=u_{m+2}=\cdots=u_n=0$. 
Then $J(u_1,\ldots,u_n)=-\infty$ and if $\kappa(P)$ was finite dimensional by JBC we have $\trdeg_K(\kappa(P))<-\infty$ which is impossible. 
This proves $\trdeg_K^{\partial}(\kappa(P))>0$. 
This form of the dimension conjecture is equivalent to the form we stated.
\label{footnote:JBC-implies-DC} } We prove the converse.

\begin{theorem}
  Suppose that $K$ is a differentially closed field.
  Suppose that the Dimension Conjecture is true. 
  Then the Jacobi Bound Conjecture is true.
\end{theorem}

\subsection{Strategy of proof}

An inductive approach to the Jacobi Bound Conjecture is the following. 
Let $\Sigma \subset \AA^n_{\infty}$ be a system given by 
$$ \Sigma \colon \quad u_1=0, \quad u_2 =0, \quad \cdots, \quad u_n=0 $$
where $u_1,\ldots,u_n \in K\lbrace x_1,\ldots,x_n\rbrace$, and let $P \supset [u_1,\ldots,u_n]$ be a minimal prime, corresponding to an irreducible component $\Sigma_1 \subset \Sigma$.
One would like to replace (or modify) this with a simpler system 
$$ \Sigma' \colon \quad v_1=0, \quad u_2 =0, \quad \cdots, \quad v_n=0 $$
such that $J(v_1,\ldots,v_n)\leq J(u_1,\ldots,u_n)$, and then induct. 

To carry out the induction, one would like $P$ to still be a minimal prime of the new system.
A natural (but not the only) way to do this is via Ritt division (Definition \ref{D:ritt-division}).
For example, if $u_2$ is ``lower'' than $u_1$ (e.g., if the highest derivative of $x_1$ occurring in $u_1$ is larger than the highest derivative of $x_1$ occurring in $u_2$),
Ritt division gives 
  \begin{equation}
  \label{eq:mini-separant}
su_1=Q(u_2)+v_1,	  
\end{equation}
where $s, v_1 \in K\lbrace x_1,\ldots,x_n\rbrace$, $Q$ is a differential operator, and $v_1$ is lower than $u_1$ (in the sense of Definition \ref{D:ranking}).

Then, it is natural to replace the system $[u_1,u_2,\ldots,u_n]$ with $[v_1,u_2,\ldots,u_n]$.
Since $u_1,Q(u_2) \in P$, it is also true that $v_1 \in P$, and thus $P \supset [v_1,u_2,\ldots,u_n]$.
\medskip

\noindent \underline{\emph{Main issues}}: to induct from here we would need the following to be true:
\begin{enumerate}
\item $J(v_1,u_2,\ldots,u_n) \leq J(u_1,u_2,\ldots,u_n)$, and 
\item $P \supset [v_1,u_2,\ldots,u_n]$ is still a minimal prime.
\end{enumerate}
Neither of these are true in general, and the bulk of this paper deals with these two issues.
\medskip  

\noindent \underline{\emph{Adapting Ritt's Proof of the Linear Case To The Nonlinear Situation}}:
Since $v_1$ is ``lower'' than $u_1$, one might expect that the Jacobi number does not increase. This is quite false; see e.g.~Example \ref{E:J-increasing-second-form}. 

Revisiting Ritt's proof of the linear case of JBC \cite[pg 310]{Ritt1935}, we see that Ritt is running a similar ``reduction algorithm'', but with big restrictions on the allowed divisions.
In particular, Ritt identified two situations in the linear case where one can perform a Ritt division without increasing the Jacobi number. 
In the present paper we explain how to adapt these two procedures to the nonlinear case unconditionally and fix a large gap in Ritt's proof. 
See Propositions~\ref{P:ritt-first} and~\ref{P:ritt-second}.
We call the two situations \defi{Ritt's first form} and \defi{Ritt's second form} (Definition~\ref{D:ritt-forms}).
\medskip

\noindent \underline{\emph{Shifting components}}:
The second issue, more generally, is that when modifying the equations defining a system, the irreducible components of the new system can change.
In particular, while we only make manipulations such that $P$ is still a prime of the new system, $P$ is generally no longer a \emph{minimal} prime of the new system.
\smallskip

This is a much more serious obstacle than it might seem at first.
\smallskip

An easy case is handled by Lemma \ref{lemma:nonvanishing-separants}: if the separant $s$ (i.e., the ``$s$'' from Equation \eqref{eq:mini-separant}) does not vanish generically on $\Sigma_1$ (i.e., $s \not \in P$), then all is well: our operation is ``locally invertible'', and $P$ is still a minimal prime of the new system. Linear systems of differential equations satisfy Lemma \ref{lemma:nonvanishing-separants} (since their separants are constants), which allows us to recover Ritt's proof of the linear case of JBC \cite[pg 310]{Ritt1935} (see Section \ref{S:linear-case}).
\medskip

\noindent \underline{\emph{Degenerate Situations}}:
Much more difficult is the case when $s \in P$ (i.e., the separant vanishes generically on the component in question). We call this a \defi{degenerate situation} (Definition~\ref{D:degenerate}).  While $P$ may no longer be minimal, it still contains \emph{some} minimal prime $P \supset P' \supset [v_1,\ldots,v_n]$ of the new system, with corresponding irreducibles $\Sigma_1 \subset \Sigma'$, whence $\dim \Sigma_1 \leq \dim \Sigma'$.
At first glance it seems sufficient to induct using $P'$; i.e., JBC should still imply that $\dim P' \leq J(v_1,\ldots,v_n)$, which would give
$$\dim \Sigma_1 \leq \dim \Sigma'\leq J(v_1,\ldots,v_n) \leq J(u_1,\ldots,u_n).$$
The persistent issue in many arguments about JBC is that there is no guarantee that $P'$ is finite dimensional (which is part of the hypothesis of JBC). 

There are only a few cases (that we know of) where one can prove finiteness of $\dim \Sigma'$ directly.
The ``easiest'' case is $n = 1$ (see Lemma \ref{L:one-variable-case}).

Another fairly straightforward case is when the order matrix is ``tropical upper triangular'', i.e., of the form			
				$$ A = \begin{pmatrix}
				a_{1,1} & a_{1,2} & \cdots & a_{1,n} \\
				-\infty & a_{2,2} & \cdots & a_{2,n} \\
				\vdots & \vdots & \ddots & \vdots \\
				-\infty & -\infty  & \cdots & a_{n,n}
			\end{pmatrix}. 
			$$
This case follows from a straightforward induction, similar to the base case of the inner induction in our proof of Theorem \ref{T:dc-implies-jbc}.
\medskip

\noindent \underline{\emph{Cylinder trick}}:
A much more exotic case (identified by Ritt in his proof of the two variable case of JBC \cite{Ritt1935}) is when $\Sigma_1$ is contained in an ``intersection of cylinders''.
By this, we mean that there exist single (differential) variable polynomials $M_1(x_1),\ldots,M_n(x_n) \in P$.
In this case, $\dim \Sigma_1$ is at most the dimension of the irreducible components of $[M_1(x_1),\ldots,M_n(x_n)]$, which are visibly finite dimensional (explained at the end of the proof of Theorem \ref{T:moving}).

It is not straightforward to manufacture such polynomials $M_i$.
Ritt cleverly does so in his proof of the two variable case of JBC in \cite{Ritt1935}.
First, instead of replacing $u_1$ with its remainder, we replace it with its separant and consider the system $[s,u_2,\ldots,u_n]$.
This has merits: since $s$ is ``visibly'' lower than $u_1$, 
\[
J(s,u_2,\ldots,u_n)\leq J(u_1,u_2,\ldots,u_n),
\]
and since $s \in P$ (by the hypothesis of our ``degenerate situation''), $P \supset [s,u_2,\ldots,u_n]$. 

Still, the issue remains that $P$ may no longer be a minimal prime. 
\medskip

\noindent \underline{\emph{Degenerations and Ritt's Pencil Trick}}:
To deal with this issue, Ritt introduces a degeneration. Instead of considering just $s$, he identifies a natural third polynomial $t$ such that $t \in P$ (see Equation \eqref{eq:ritt-pencil-s-t}). Then for any $\mu \in K$, $t + \mu s \in P$. Varying $\mu$ gives a family $[t + \mu s,u_2,\ldots,u_n]$ of differential ideals, all of which are contained in $P$. Ritt shows that for most choices of $\mu$, $P$ contains a minimal prime $P_{\mu}$ that is finite dimensional. The proof that $P_{\mu}$ is finite dimensional is via the above cylinder trick; i.e., Ritt is able to produce sufficient $M_i(x_i)'s$ in $P_{\mu}$.

We call this technique the \emph{Ritt Pencil Trick}.
\medskip

\noindent \underline{\emph{Generalizing Ritt's Pencil Trick}}:
One of the main contributions of this paper is that if we assume the Dimension Conjecture, then Ritt's pencil trick can be extended to arbitrarily many variables (with a lot of work) to similarly bypass such ``degenerate situations''.

This pencil is a differential algebraic pencil of differential subschemes of $\AA^n_{\infty}$, which we call the \defi{Ritt Pencil} $\Gamma$ (Definition~\ref{D:ritt-pencil}).
This is essentially the family of all of the different reasonable ways we can reduce the equations for $\Sigma_1$ to something lower, rather than just replacing an equation with its separant. 
It is a differential algebraic map
$$\Gamma \to \AA^1_{\infty} = \Spec K\lbrace y \rbrace.$$ 
Note that both sides are non-Noetherian and have infinite Krull dimension.

For the Ritt pencil (Figure~\ref{F:ritt-pencil}) we can show that the original component $\Sigma_1$ sits inside the base locus, and that the base locus is contained in our original differential algebraic variety $\Sigma$:
 $$ \Sigma_1 \subset \BS(\Gamma) \subset \Sigma.$$
Then we perform an analysis of the component of $\Gamma$ containing $\Sigma_1$. 
This component can be moving or fixed and we need to treat each of these cases separately. 
This is carried out in Sections~\ref{S:linear-series} and \ref{S:ritt-pencil}.
The key application of the dimension conjecture is to show that an appropriate component of a fiber $\Gamma_{\mu}$ has finite absolute dimension, again via the cylinder trick.

\begin{figure}[htbp!]
	\begin{center}
		\includegraphics[scale=0.7]{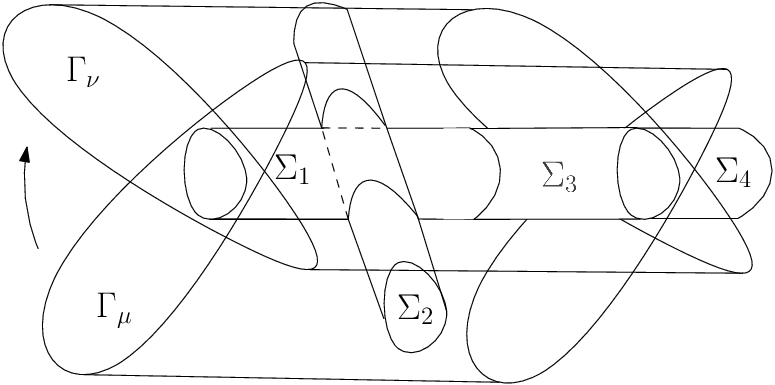}
	\end{center}
	\caption{A Ritt pencil. We have written $\Sigma$ as a union of irreducible component $\Sigma_1,\Sigma_2,\Sigma_3,\Sigma_4$. 
	The fibers $\Gamma_{\mu}$ and $\Gamma_{\nu}$ of the pencil $\Gamma$ have an intersection which contains a base locus. The base locus is contained in $\Sigma$ and contains $\Sigma_1$. }\label{F:ritt-pencil}
\end{figure}

To highlight the novelty of this, we remind the reader that, in differential algebraic geometry, we don't have projective spaces, we have no theory of line bundles to lean on, no cohomological vanishing theorems like Serre vanishing, we do not have flatness for these families, our rings are not finite type over a field, and they aren't even Noetherian. 

Of particular interest is the case where $\Sigma_1$ is contained in a component of $\Gamma$ that is moving.
Ritt's argument here is very specific to the two variable case, and his approach does not generalize to more variables. In Section~\ref{section:Dimension-Conjecture-Implies-Jacobi-Bound-Conjecture} we explain how to generalize this argument using the dimension conjecture together with some differential algebraic elimination theory. 
It is worth mentioning that at the time of \cite{Ritt1935}, the theory of characteristic sets was not developed; in particular, the elimination theory that we use was not available to him.
\medskip

Finally, we would like to highlight a number of things we have developed along the way.
\medskip

\noindent \underline{\emph{The Dimension Conjecture}}:
The statements of the Dimension Conjecture found in Ritt \cite[pg 178]{Ritt1950} and Cohn \cite[pg 1]{Cohn1983} are not the same. 
They also differ from the statement of the conjecture that we will want to use for our application, which closely resembles Krull's Principal Ideal Theorem (PIT) in Commutative Algebra. 
Since we could not find a reference that proves that these variants are equivalent, in Section \ref{S:dimension-conjecture} we supply a proof that the three versions of the dimension conjecture are equivalent to each other. 
In the same section we apply the PIT version to derive a statement we will need later in our analysis of moving components of the Ritt Pencil.
\medskip 

\noindent \underline{\emph{Ritt Cycle Trick}}: We have isolated a useful group-theoretic tool which Ritt uses which we call the ``Ritt cycle trick'' (Lemma \ref{lem:ritt-permutation}) which applies in the general nonlinear situation. 
	This is used to show that Jacobi bounds don't increase when performing certain Ritt divisions.
\medskip

\noindent \underline{\emph{A Gap in Ritt's Proof Of The Linear Case}}: Finally, while adapting Ritt's linear reductions to the nonlinear situation the authors identified and repaired a rather large gap in Ritt's original proof of the linear case from \cite{Ritt1935} (see also \cite{Ritt1950}). We also fix several minor gaps throughout his paper in our proofs of Propositions~\ref{P:ritt-first} and~\ref{P:ritt-second}).
	
	The fix in this paper appears over several pages as Proposition \ref{P:second-form}.
	The main gap is that he assumed that the order matrices of a certain shape can always be brought into a standard form which we call \defi{Ritt's Second Form}. 
	Fixing this gap was a lot of work, which surprisingly used a clever application of Hall's Matching Theorem from graph theory.
	Ritt writes at the end of the first paragraph of \cite[P.~309]{Ritt1935} 
	\begin{quote}
		``...that (c) can be realized is now evident.''
	\end{quote}
	where $(c)$ in loc.~cit.~is the statement that the matrix can be brought into Ritt's second form. No further explanation is given.

	The appearance of Hall's Matching Theorem is particularly interesting in that, while Ritt's work predates Hall's matching theorem, a variant called Egerv\'{a}ry's Theorem, which is equivalent to Hall's Matching Theorem \cite[Section 2.1]{Diestel2025}, appears in Ollivier's work on the Jacobi Bound Conjecture \cite[Section 2]{Ollivier2022} and in Cohn's work \cite[section 5]{Cohn1983}.  
	Ollivier uses Egerv\'{a}ry's Theorem in his algorithms for computing Jacobi bounds, and according to Ollivier in \cite{Ollivier2022},  Egerv\'{a}ry's Theorem  was also investigated by Jacobi in his original manuscript.
	It is unclear what our application to Proposition~\ref{P:second-form} has to do with these other applications;
  it is curious that this piece of combinatorics seems to have resurfaced in the context of the Jacobi Bound Conjecture in several different places.

  \subsection{The Rosenfeld--Gr\"{o}bner algorithm}

  The reduction process we exploit in this paper is similar to the Rosenfeld--Gr\"{o}bner algorithm \cite[Algorithm 1]{Golubitsky2007}.
  For a given system $[u_1,\ldots,u_n]$, this algorithm produces, for every minimal prime $P \supset [u_1,\ldots,u_n]$, a characteristic set\footnote{Technically  \cite[Algorithm 1]{Golubitsky2007} produces some slightly weaker autoreduced sets called ``regular differential systems'', but one can enhance the algorithm to produce characteristic sets; for the purposes of this introduction it doesn't matter. See \cite{Ovchinnikov2004}, \cite[See Algorithm 7.1]{Hubert2000}.}
(defined in Subsection \ref{ss:characteristic-sets}) $G=(g_1,\ldots,g_s)$ for $P$.
  
In our setting, it is interesting to consider a characteristic set for our given component $\Sigma_1$; in \cite[Corollary 1.8]{Dupuy2026a} we prove that if $G=(g_1,\ldots,g_n)$ is a characteristic set for $\Sigma_1$, then $\dim(\Sigma_1) = J(g_1,\ldots,g_n)$.
One could imagine a proof of JBC where one analyzes each step of the algorithm and hopes e.g.~that the Jacobi number never increases.

Unfortunately, a cursory initial investigation reveals that this approach seems doomed to fail.
There are several issues.
First, during the Rosenfeld--Gr\"{o}bner algorithm one can pass from a situation with $n$ equations to a situation with more than $n$ equations; the Jacobi Bound Conjecture doesn't apply to systems with more than $n$ equations in $n$ variables. 
Second, the algorithm involves Ritt divisions that are not in Ritts first or second form,
and even in the linear case, there exist examples where the Jacobi bound increases as Example~\ref{E:J-increasing} shows.

\begin{example}\label{E:J-increasing}
Consider the system of differential equations 
$$\begin{cases}
	u_1=x^{(100)}+y'+z',\\
	u_2=x^{(50)}+y+z, \\
	u_3=x'+y'+1.
\end{cases}$$
One can perform the Rosenfeld--Gr\"{o}bner algorithm and track the Jacobi bound of the resulting system after each division is performed.\footnote{See \url{https://github.com/tdupu/magma-diff-alg} for an open source Magma package.}  
The initial Jacobi bound is $J_0=101$ and we find subsequently at each stage that $J_1=150, J_2=101, J_3=101$, and thereafter it stabilizes at 101.
\end{example}

While the branching issue appears to be a fundamental problem, the increasing Jacobi bound issue appears to improve by passing to strong Jacobi bounds.
Here the order of a differential polynomial is declared to be $-\infty$ in a differential variable if the variable is absent from the differential polynomial (notation in \S\ref{S:notation}). 
In passing to the strong bound we are asking for a tighter inequality which should make the problem harder. Counterintuitively, the strong bounds make the induction more restrictive and hence easier to induct.
\begin{example}
	Consider the system $u=x'+y^{(18)}$, $v=(y')^2+y$. 
	Then we have order matrices with the weak and strong conventions given by 
	$$ \begin{pmatrix}
		1 & 18 \\ 0 & 1
	\end{pmatrix}, \quad \begin{pmatrix} 1 & 18 \\ -\infty & 1\end{pmatrix}. $$
	In the first case the Jacobi bound is 18 while in the second case the presence of $-\infty$ in the first column forces the Jacobi bound to be 2. 
\end{example}

Unfortunately, even passing to the strong bound does not fix the situation with reductions increasing the Jacobi Bound. In Example~\ref{E:J-increasing-second-form} we provide more examples where division can increase the Jacobi number, even in the strong case (we performed many Magma experiments along these lines).
This pretty much dooms the approach of tracing through the divisions in the Rosenfeld--Gr\"{o}bner algorithm.

\subsection{Acknowledgements}
Dupuy is supported by National Science Foundation grant DMS-2401570.
Zureick-Brown is supported by National Science Foundation grant DMS-2430098.

Dupuy is grateful for the hospitality of the GAATI laboratory at the l'Université de la Polynésie française, where he completed this manuscript.

\subsection*{Software}
The authors used Magma \cite{Magma}, including a differential algebra package \cite{dupuyZB:magma-package}, at various points in this project. 

GPT5  \cite{OpenAI2026} was used to perform literature searches and provide numerical examples, proofs and verifications of individual claims in the paper, and initial code for the figures used. 
However, the final arguments in this paper are human-generated.

Significantly, the idea to use Hall's Matching Theorem to resolve the gap in Ritt's proof of the linear case of the Jacobi Bound Conjecture (Proposition \ref{P:second-form}) was assisted by GPT5.
We gave an unsatisfying proof of a small subcase to GPT5, and it responded with a very clean proof of the same special case using Hall's Matching Theorem.
GPT5 was not able to give a full proof of Proposition \ref{P:second-form}, but after a lot of additional work we were able to see how to apply Hall's Matching Theorem to give a full proof.

\section{Background}

In this paper, when $(\Gamma,<)$ is an ordered set we write $<_{\lex}$ for the induced \defi{lexicographic order} on $\Gamma^n$. 

\subsection{Notation and Conventions For Schemes}

For a scheme $X$ we will let $\Irr(X)$ denote the irreducible components of $X$. 
By $\Irr^0(X)$ we mean the minimal irreducible components. 
If $X = \Spec(A)$ and $\sqrt{\langle 0 \rangle_A} = \bigcap_{i \in I} P_i$ is an irredundant intersection of minimal prime ideals, then we can write $Q_i = \sat_{P_i}( \langle 0 \rangle_A)$ for the associated $P_i$-primary component, in which case $\Irr^0(X) = \lbrace V(Q_i) \colon i \in I \rbrace$. 
In this paper we only deal with minimal irreducible components which correspond to minimal associated primes of ideals and not embedded primes.
Equivalently, these are components of the underlying reduced scheme, and algebraically, these are minimal primes of the underlying radical ideal.

If $X_1$ is an irreducible component of $X$, then we will let $\kappa(X_1)$ denote the function field of $(X_1)_{\red} $ the underlying reduced irreducible scheme (algebraically, if $X=\Spec(B)$, then $X_{\red} = \Spec(B/\sqrt{\langle 0 \rangle_{B}})$). 
Again, if $R$ is a ring and $P$ is a prime ideal we will let $\kappa(P)$ denote the fraction field of $R/P$. 
If $\Spec(R)$ is irreducible we may also write $\kappa(R)$ for $\kappa(\Spec(R))$.

We can now define the dimension either in terms of transcendence degrees or Krull dimension. 
These two notions coincide for schemes which are finite type over a field. 
Our rings, in general, are countably generated and therefore are generally not Noetherian. 
Since the rings we care most about have localizations which are finite type \cite[Proposition 3.2]{Dupuy2026a} we define our dimension via transcendence degrees.
\begin{definition}
Let $X$ be a scheme over a field $K$. 
The \defi{dimension} of $X$ over $K$ is  
$$\dim_K(X) = \max \lbrace \trdeg_K(\kappa(X_1))\colon X_1 \in \Irr^0(X) \rbrace.$$ 
\end{definition}
We will often omit the $K$ and just write $\dim(X)$. 

\begin{remark} 
  Warning: if $R$ is a domain and $P$ is a prime ideal, then $\dim(\Spec(R_P)) = \dim(\Spec(R))$, but $\Krull(\Spec(R_P)) \neq \Krull(\Spec(R))$ in general. 
  In particular, our definition of dimension does not generally coincide with Krull dimension.
\end{remark}

\subsection{Notation and Conventions for Differential Schemes}
We follow \cite[Ch3]{Buium1994} for our terminology in differential algebraic geometry. 
Let $(K,\partial)$ be a differential field. 
\begin{definition}
	By a \defi{$D$-scheme} over a differential field $K$ we mean a scheme $X$ such that its structure sheaf $\Ocal_X$ is a sheaf of differential $K$-algebras.
\end{definition}
We will use the words \defi{differential scheme} and \defi{$D$-scheme} interchangeably. 
A morphism of $D$-schemes over $K$ is a morphism of schemes over $K$ over such that the induced map of sheaves is a morphism of sheaves of differential $K$-algebras. 
Usual adjectives for objects and morphisms in algebraic geometry (e.g.~reduced, projective, affine, surjective etc.) apply to the underlying scheme or morphism of schemes.

If $\Sigma$ is an affine $D$-scheme of $\partial$-finite type defined by an ideal $I \subset K\lbrace x_1,\ldots,x_n\rbrace$, then by an irreducible component $\Sigma_1\in \Irr^0(\Sigma)$ we mean $\Sigma_1 = V(Q_1)$, where $Q_1=\sat_{S_{P_1}}(I)$ is the saturation of $I$ by $S_{P_1}$, where $S_{P_1}$ is the multiplicative set associated to a differential prime $P_1 \supset I$ minimal above $I$.  This is automatically a $D$-scheme since the localization map $K\lbrace x_1,\ldots,x_n\rbrace/I \to S_{P_1}^{-1}( K\lbrace x_1,\ldots,x_n\rbrace/I )$ is a morphism of differential $K$-algebras and its kernel is $Q_1$. 
For a differential ring $R$ we let $\Spec^{\partial}(R)$ denote the collection of differential prime ideals. 

If $\Sigma$ is an irreducible $D$-scheme, then $\kappa(\Sigma)$ is a $\partial$-field. 
In particular, if $P$ is a prime ideal in a $\partial$-ring $R$, then $\kappa(P)$ is a $\partial$-field. 

\begin{definition}
	Let $R$ be a differential field.
	We say that a differential $R$-algebra $A$ is \defi{differentially of finite type} if there exists an integer $m$ and a surjective morphism of differential $R$-algebras $R\lbrace x_1,\ldots, x_m\rbrace \to A$. 
\end{definition}
We will work over fields and simply write ``$A$ is $\partial$-finite type'' when $A$ is $\partial$-finite type over a differential field $K$.

\begin{definition}
	Let $R$ be a differential $K$-domain. 
	The \defi{differential transcendence degree} of $R$ over $K$, denoted by $\trdeg^{\partial}_K(R)$, is the largest $n$ such that there exists an injection of differential $K$-algebras $K\lbrace x_1,\ldots,x_n\rbrace \to R$. 
\end{definition}

\begin{definition}
	For a $D$-scheme $\Sigma$ over a differential field $K$ which is locally of $\partial$-finite type we define the \defi{differential dimension} by
	$$\dim^{\partial}_K(\Sigma) = \max \lbrace \trdeg_K^{\partial}(\kappa(\Sigma_1))\colon \Sigma_1 \in \Irr^0(\Sigma) \rbrace.$$ 
\end{definition}
\noindent The dimension $\dim_K(\Sigma)=\dim(\Sigma)$ will just be the dimension of the underlying scheme (defined via transcendence degrees of components).

\begin{definition}
	By a \defi{differential point} of a $D$-scheme $\Sigma$ we mean a morphism of $D$-schemes $\Spec(R) \to \Sigma$ where $R$ is a differential ring and $\Spec(R)$ is given its $D$-scheme structure. 
\end{definition}
\noindent Differential points correspond to solutions of differential equations in differential rings.
\\

\noindent Given $n\geq 0$ we will view the symbol $\AA^n_{K,{\infty}}$ as a $D$-scheme over $K$ so that 
$$ \AA^n_{K,{\infty}} = \Spec(K\lbrace x_1,\ldots,x_n \rbrace).$$
This is the infinite jet bundle (also called the arc space) of $\AA^n_K=\Spec K[x_1,\ldots,x_n]$.
For a scheme $X$, the infinite jet bundle is denoted $\mathcal{L}(X)$ \cite{Cluckers2011}, $X_{\infty}$  \cite{Fernex2015}, or $J^{\infty}(X)$ or $X^{\infty}$ \cite{Buium1994}.
If $R$ is differential $K$-algebra and if $\AA^n_{K,{\infty}}(R)$ denotes set of morphisms of $D$-schemes $\Spec(R)\to \AA^n_{\infty}$, then $\AA^n_{K,{\infty}}(R)=R^n$.
Because of this we think of $\AA^n_{K,{\infty}}$ as just affine space in the differential algebraic setting.

\subsection{The Jacobi Bound Conjecture}\label{S:notation}

Let $K$ be a differential field of characteristic zero. 
Let $K\lbrace x_1,\ldots,x_n\rbrace$ be the ring of differential polynomials over $K$. 

\begin{definition}
	Let $u \in K\lbrace x_1,\ldots,x_n\rbrace$. We define the \defi{order} of $u$ by
	$$ \ord_{x_j}^{\partial}(u) = \max \lbrace r \in \ZZ_{\geq 0}  \colon \partial u/\partial x_j^{(r)} \neq 0  \rbrace.$$
	In the \defi{strong convention} the maximum of the empty set is defined to be $-\infty$ and in the \defi{weak convention} the maximum of the empty set is defined to be $0$. 
\end{definition}

\begin{definition}
	Let $u_1,u_2,\ldots,u_n\in K\lbrace x_1,\ldots,x_n\rbrace$ be differential polynomials. 
	The \defi{order matrix} of $(u_1,\ldots,u_n)$ is the matrix $A = (a_{i,j})_{1\leq i \leq n, 1\leq j \leq n}$ where $a_{i,j} = \ord_{x_j}^{\partial}(u_i)$.
\end{definition}

\begin{definition}
  A \defi{transversal} of an $n\times n$ matrix $A$ is a sum of the form $\sum_{i=1}^n a_{i,\rho(i)}$, where $\rho \in S_n$.
  \end{definition}

\begin{definition}
	Let $u_1,u_2,\ldots,u_n\in K\lbrace x_1,\ldots,x_n\rbrace$ be differential polynomials. 
	We define the \defi{Jacobi bound} $J(u_1,\ldots,u_n)$ to be $\tdet(A)$, where $\tdet$ denotes the tropical determinant of the matrix $A \in M_{n}(\ZZ_{\geq 0} \cup \lbrace - \infty \rbrace)$.
	It is given by $\tdet(A) = \max_{\rho \in S_n } \sum_{i=1}^n a_{i,\rho(i)}$, i.e., a maximal transversal.
  \end{definition}

We call the Jacobi bound \defi{weak} or \defi{strong} depending on whether we are using the weak or strong convention for the order of differential polynomials with respect to a variable.  
\begin{conjecture}[Jacobi Bound Conjecture]\label{C:jbc}
	Let $K$ be a differential field of characteristic zero. 
	Let $u_1,u_2,\ldots,u_n\in K\lbrace x_1,\ldots,x_n\rbrace$.
	Let $\Sigma=\Spec K\lbrace x_1,\ldots,x_n \rbrace/[u_1,\ldots,u_n]$.
	Let $\Sigma_1$ be an irreducible component of $\Sigma$.
	If $\dim(\Sigma_1)<\infty$, then $\dim(\Sigma_1)<J(u_1,\ldots,u_n)$. 
\end{conjecture}
The Jacobi Bound Conjecture is called \defi{weak} or \defi{strong} according to whether we are using the weak or strong conventions for the Jacobi bound.
Both the weak and strong Jacobi Bound Conjectures are open and the strong Jacobi Bound Conjecture implies the weak Jacobi Bound Conjecture simply because the strong Jacobi bound is less than or equal to the weak Jacobi bound. 

\subsection{Characteristic Sets}
\label{ss:characteristic-sets}

\begin{definition}
A term ordering $\prec$ on the infinite polynomial ring $K\lbrace x_1,\ldots,x_n\rbrace$ is called a  \defi{ranking} if the following properties hold:
\begin{enumerate}
	\item For every $i$ with $1\leq i \leq n$ and every $r$ with $r\geq 0$ we have $x_i^{(r+1)}\succ x_i^{(r)}$. 
	\item For every $i,j$ with $1 \leq i, j \leq n$ and $r,s\geq 0$ if $x_i^{(r)} \succ x_j^{(s)}$, then $x_i^{(r+1)} \succ x_j^{(s+1)}$.
\end{enumerate} 
\end{definition}
We say a ranking $\prec$ is an \defi{elimination ordering} if there exists a partition of the $X = \lbrace x_1,\ldots,x_n\rbrace = Y \amalg Z$ such that for all $y \in Y$, for all $z \in Z$, and for all $r,s\geq 0$, we have $y^{(r)} \prec z^{(s)}$. 
This is the same as for usual polynomial ring except that our blocks are infinite.
This ordering has the property that if $f\in K\lbrace X \rbrace$ has its leading monomial in $K\lbrace Y \rbrace$ then $f \in K\lbrace Y \rbrace$.

\begin{definition}
	Let $u \in K\lbrace x_1,\ldots,x_n\rbrace$.
In this paper, if $\ord_{x_j}^{\partial}(u)=r_j\geq 0$, then we will call $x_j^{(r_j)}$ the \defi{leader of $u$ in the variable $x_j$} and $\partial u/\partial x_j^{(r_j)}$ the \defi{separant of $u$ in the variable $x_j$}. 
When no $x_i$ is specified the \defi{leader} of an element $u$ is the largest $x_i^{(r)}$ in the term ordering such that $\partial u/ \partial x_i^{(r)} \neq 0$. 
We denote it by $\ell_u$.
If $\ell_u=x_i^{(r)}$, then the \defi{leading variable} of $u$ is $x_i$. 
If the words leader or separant are used without reference to a variable we always mean the leader and separant with respect to the leading variable. 
\end{definition}

\begin{definition}
\label{D:ranking}  
Let $f,g \in K\lbrace x_1,\ldots,x_n\rbrace$. 
We say that $f$ is \defi{larger} that $g$ in the differential indeterminate $x_i$, and write $f \succ_i g$, if and only if $\ord^{\partial}_{x_i}(f)>\ord^{\partial}_{x_i}(g)$, or $\ord^{\partial}_{x_i}(f)=\ord^{\partial}_{x_i}(g)$  (i.e., $\ell$ is the simultaneous leader of $f$ and $g$ in the variable $x_i$) and, writing $f=\sum_{i=1}^{d_f} a_i \ell^i$ and $g=\sum_{i=1}^{d_g} b_i \ell^i$, we have $d_f>d_g$. 

If no variable is mentioned, and say \defi{$f$ is larger than $g$} we mean that $\operatorname{in}_{\prec}(f)\succ \operatorname{in}_{\prec}(g)$ where $\operatorname{in}_{\prec}$ denotes the leading monomial in the term order. 
\end{definition}

\begin{definition}
	Let $R$ be a $\partial$-ring. 
	The \defi{Weyl algebra} of $R$ is the unique associative algebra $R[\partial]$ which is a left $R$-algebra, as a left $R$-module is a direct sum $R[\partial] = \bigoplus_{i \geq 0} R\partial^i$, and such that for every $r\in R$, we have the commutation relation $\partial r = \partial(r) + r\partial$.
\end{definition}

\begin{definition}
	Let $R$ be a differential ring. 
	We say that an $R$-module $M$ is a \defi{$D$-module} if it is a left $R[\partial]$-module. 
\end{definition}

The ring $R$ itself is a $D$-module and ideals $I \subset R$ which are $D$-submodules are differential ideals. Given $L \in R[\partial]$ and $r\in R$ we will write $L(r)$ for the $D$-module multiplication.
In the $D$-module literature authors often write $\mathcal{D}=R[\partial]$ so that a $D$-module quite literally becomes a left $\mathcal{D}$-module.

\begin{definition}
\label{D:ritt-division}  

Given two differential polynomials $f$ and $g$ in $K\lbrace x_1,
\ldots,x_n\rbrace$ and a variable $x_i$ such that $g$ is lower than $f$ in the variable $x_j$, we can perform a \defi{Ritt division}.  This yields an expression of the form 
\begin{equation}\label{E:ritt-division}
sf=Q(g)+r	
\end{equation}
where $Q\in K\lbrace x_1,\ldots,x_n\rbrace[\partial]$, $s$ is in the multiplicative set generated by the initial and separant of $g$ with respect to the variable $x_i$, and $r$ is lower than $g$ in the variable $x_i$  \cite[Ch2, \S2, pg25]{Buium1994}.
\end{definition}

There is also \defi{Ritt $\partial$-division} (read ``Ritt partial division'') where one does not ``fully reduce'' but only finds $r$ such that $\ord^{\partial}_{x_i}(r) \leq \ord^{\partial}_{x_i}(g)$. 
In this case, $s$ in equation \eqref{E:ritt-division} can be taken to be in the multiplicative set generated by the separant of $g$ in the variable $x_i$ (so initials are not used). 
See \cite[Algorithm 2.2]{Hubert2000}.

As before, if no variables are mentioned and $f\succ g$, then the division is assumed to be in leading variable of $g$ with respect to $\prec$.

\begin{definition}
If, after a division of $f$ by $g$ (resp.~$\partial$-division, resp.~division with respect to $x_i$, resp.~$\partial$-division with respect to $x_i$), one finds that $r=0$ in \eqref{E:ritt-division}, then we say that $f$ is \defi{Ritt reduced} (resp.~\defi{Ritt $\partial$-reduced}, resp.~\defi{Ritt reduced in $x_i$}, resp.~\defi{Ritt $\partial$-reduced in $x_i$}) with respect to $g$. 
\end{definition}

\begin{definition}
An \defi{autoreduced sequence} (also called an \defi{autoreduced set}) is a sequence $g_1\prec g_2 \prec \cdots \prec g_s$ with $g_i \in K\lbrace x_1,\ldots,x_n\rbrace$ such that for all $i$ and $j$ with $1\leq i,j \leq s$ and $i\neq j$ we have that $g_i$ is Ritt reduced with respect to $g_j$.
We will write such a sequence as $G=(g_1,\ldots,g_s)$.
\end{definition}

For autoreduced sequences there exist similar notions of Ritt division with the various adjectives which we call \defi{Ritt division by $G$}.
See \cite[pg 25, Ch2, \S2]{Buium1994}. 
In these situations the division algorithm yields an expression
\begin{equation}\label{E:ritt-division-autored}
	sf=Q_1(g_1)+\cdots + Q_r(g_r)+r	
\end{equation}
where $Q_1,\ldots,Q_r \in K\lbrace x_1,\ldots,x_n\rbrace[\partial]$ and $s \in S_G$, where $S_G$ is the multiplicative set generated by the initials and separants of $G$ (in the $\partial$-division case, $S_G$ can be taken to just be generated by the separants of elements of $G$), and $r$ is Ritt reduced with respect to each $g_i$ in $G$. 
We use the notation $r=\rem_G(f)$ for the Ritt remainder after Ritt division by $G$. 

\begin{definition}
Let $\prec$ be a ranking on $K\lbrace x_1,
\ldots,x_n\rbrace$.
Let $G=(g_1,\ldots,g_s)$ and $H=(h_1,\ldots,h_{s+t})$ be two autoreduced sets where $t \in \ZZ_{\geq 0}$.  
The \defi{induced ordering} $\prec$ on autoreduced sets is given by 
 $$G \prec H \iff (g_1,\ldots,g_s) \prec_{\lex} (h_1,\ldots,h_s)$$
and 
 $$H \prec G \iff (h_1,\ldots,h_s) \prec_{\lex} (g_1,\ldots,g_s) \mbox{ or } ((g_1,\ldots,g_s) \nprec_{\lex} (h_1,\ldots,h_s) \mbox{ and } t>0).$$
\end{definition}

\begin{definition}\label{D:characteristic-set}
	Fix a ranking $\prec$ on $K\lbrace x_1,\ldots,x_n\rbrace$. 
	Let $P \subset K\lbrace x_1,\ldots,x_n\rbrace$ be a prime differential ideal. 
	A \defi{characteristic sequence} (also called a \defi{characteristic set}) for $P$ in the term order $\prec$ is an autoreduced sequence $G=(g_1,\ldots,g_s)$ with $g_i \in P$ which is minimal in the induced ordering on autoreduced sets. 
\end{definition}

\begin{definition}
Let $R$ be a ring and $I$ an ideal in $R$. 
If $S \subset R$ is a multiplicatively closed set and $\ell_S\colon R \to R_S = S^{-1} R$ is the localization-at-$S$ map, then the \defi{saturation of $I$ with respect to $S$} is the ideal $$\sat_S(I) = \ell_S^{-1}( I R_S) = \lbrace r \in R \colon \exists s \in S, sr \in I \rbrace.$$ 
\end{definition}

We make use of the following fundamental properties of characteristic sets and prime ideals. 
\begin{theorem}
\label{T:fundamentals-of-char-sets}  
	Let $P \subset K\lbrace x_1,\ldots,x_n\rbrace$ be a prime ideal. 
	\begin{enumerate}
		\item For any ranking $\prec$ on $K\lbrace x_1,\ldots,x_n\rbrace$ there exists a characteristic sequence $G=(g_1,\ldots,g_s)$ for $P$.
		\item \label{item:saturation-char-set}
  If $G=(g_1,\ldots,g_s)$ is a characteristic sequence for $P$, then $P= \sat_{S}([g_1,\ldots,g_s])$.
		\item \label{item:membership-char-set} Let $f \in K\lbrace x_1,\ldots,x_n\rbrace$. 
		We have that $f \in P$ if and only if $\rem_G(f)=0$. 
	\end{enumerate}
\end{theorem}
\begin{proof}[References]
The first statement can be found in Kolchin's book \cite[I.10, Proposition 3]{Kolchin1973}.
An algorithm for the first and second statements is found in Hubert's manuscript \cite[Theorem 4.5, Theorem 5.2, and Algorithm 7.1]{Hubert2000}.
The third item can also be found in Kolchin's book \cite[III.8, Lemma 5]{Kolchin1973}.
It is commonly referred to as Rosenfeld's Lemma (see \cite[below Theorem 1]{Kondratieva2006}).
\end{proof}
A differential algebra package for working with differential polynomial rings and characteristic sets authored by the authors of this manuscript and used for experimentation can be found on github \cite{dupuyZB:magma-package}. 

\begin{proposition}\label{P:characteristic-sets-and-dimension}
	If $P$ is a prime differential ideal in $K\lbrace x_1,\ldots,x_n\rbrace$ with a characteristic set $(g_1,\ldots,g_s)$ with respect to some ranking, then $\trdeg^{\partial}(\kappa(P))=n-s$. 
	In particular, for all ranking and all characteristic sets of $P$ with respect to that ranking, the characteristic set will have $s$ elements.
\end{proposition}
\begin{proof}
	This is \cite[Section 4.1, Section 4.2, particularly Theorem 4.5]{Cluzeau2003}.

	We give a direct proof.
	Let $P$ be a prime differential ideal with characteristic set $G=(g_1,\ldots,g_s)$ in some ranking $\prec$. 
	The leading variables of the $g_i$ are all distinct, otherwise the sequence would not be autoreduced. 
	 Write $\lbrace x_1,\ldots,x_n\rbrace = \lbrace z_1,\ldots,z_s\rbrace \amalg \lbrace y_1,\ldots,y_t\rbrace$ where $\lbrace z_1,\ldots,z_s\rbrace$ are the leading variables of $(g_1,\ldots,g_s)$ and $y_1,\ldots,y_t$ are the other variables. 
	 In this situation we have that $s+t=n$.
	 For each $i$ with $1\leq i \leq t$ and every $j\geq 0$ we have that $y_i^{(j)} \notin P$ and $y_i^{(j)}$ is Ritt-reduced with respect to $G$. 
	This proves that $K\lbrace y_1,\ldots,y_t \rbrace \to K\lbrace x_1,\ldots,x_n\rbrace/P$ is injective; 
	in particular $\trdeg_K^{\partial}( K\lbrace x_1,\ldots,x_n\rbrace/P)\geq t$. 
	
	Now, over the function field $F=K(\lbrace y_1,\ldots,y_t\rbrace)$, the set $G=(g_1,\ldots,g_s)$ is a minimal autoreduced set $[g_1,\ldots,g_s] \subset F\lbrace z_1,\ldots,z_s\rbrace$ and hence a characteristic set with prime ideal $\widetilde{P}=\sat_G([g_1,\ldots,g_s])$ where the ranking on $F\lbrace z_1,\ldots,z_s\rbrace$ is induced by the ranking on $K\lbrace x_1,\ldots,x_n\rbrace$. 
	Since $g_1,\ldots,g_s$ considered over $F$ is now $s$ equations in $s$ variables, 
	by \cite[Corollary 1.8]{Dupuy2026a}, $\trdeg_F(F\lbrace z_1,\ldots,z_s\rbrace/\widetilde{P})\leq J(g_1,\ldots,g_s)$ and hence $\trdeg_F^{\partial}(F\lbrace z_1,\ldots,z_s\rbrace/\widetilde{P})=0$.
	Since $\kappa(P) = \kappa(\widetilde{P})$, we have 
  \begin{equation}
  \label{eq:trdeg-eq}  
  \trdeg_K^{\partial}(\kappa(P))=\trdeg_F^{\partial}(\kappa(P)) + \trdeg_K^{\partial}(F) = t.
  \end{equation}
	
	Since Equation \eqref{eq:trdeg-eq} is independent of the ranking, it also follows that $s$ is independent of the choice of ranking $\prec$ and characteristic set.  
  \end{proof}

	
\begin{proposition}[Prime Differential Ideal Membership]\label{P:differential-ideal-membership}
	Let $P$ be a prime differential ideal in $K\lbrace x_1,\ldots,x_n\rbrace$ with characteristic set $(A_1,\ldots,A_s)$ in some ordering. 
	Let $f\in K\lbrace x_1,\ldots,x_n\rbrace$. 
	We have $f \in P$ if and only if the remainder of $f$ with respect to Ritt division by $(A_1,\ldots,A_n)$ is zero.
\end{proposition}
\begin{proof}
	\cite[discussion below Definition 2]{Ovchinnikov2004}.
\end{proof}

The following is well-known, but we do not know of a reference which states this cleanly. 
\begin{proposition}[Elimination Theorem for Characteristic Sets]\label{P:characteristic-elimination}
	Consider a block ordering on $S = K\lbrace x_1,\ldots,x_n,y_1,\ldots,y_m\rbrace$ such that $x_i$ and all of its derivatives are higher than any derivative of any $y_j$ (i.e., $y_j^{(b)} \prec x_i^{(a)}$ for all $i,j,a,b$). 
	Let $P \subset K\lbrace x_1,\ldots,x_n,y_1,\ldots,y_m\rbrace$ be a prime ideal. 
	If $(B_1,\ldots,B_s,A_1,\ldots,A_t)$ is a characteristic set for $P$ such that $B_i \in K\lbrace y_1,\ldots,y_m \rbrace$ and $A_i \in K\lbrace x_1,\ldots,x_n, y_1,\ldots,y_m \rbrace \setminus K\lbrace x_1,\ldots,x_n\rbrace$, then $(B_1,\ldots,B_s)$ is a characteristic set for $P \cap K \lbrace y_1,\ldots,y_m\rbrace$. 
\end{proposition}
\begin{proof}
	Let $Q = P \cap K\lbrace y_1,\ldots,y_m \rbrace$. 
	Let $f \in K\lbrace y_1,\ldots,y_m\rbrace$. 	
	First, we claim that $(B_1,\ldots,B_s)$ is a minimal autoreduced subset of $Q$. 
	If not, then we could modify the autoreduced set $(B_1,\ldots,B_s,A_1,\ldots,A_t)$ to a larger one as well. 
	We will now show membership. 
	Suppose $f \in Q$. 
	Then $f \in P$, and hence its Ritt-remainder after by $(B_1,\ldots,B_s,A_1,\ldots,A_t)$ will be zero by the membership test (Proposition~\ref{P:differential-ideal-membership}). 
	But $f$ is Ritt reduced with respect to $(A_1,\ldots,A_t)$ since the leaders are higher than the leader of $f$. 
	This implies the division algorithm will only involve the $(B_1,\ldots,B_s)$ and hence $f$ is zero with respect to division by $(B_1,\ldots,B_s)$. 
\end{proof}

The following lemma is also well known, but we could not find a clean reference.

\begin{lemma}
\label{L:one-variable-case}  
Let $T \subset K\lbrace x \rbrace \setminus \{0\}$ be a finite subset, let $[T]$ be the differential ideal generated by $T$, and let $P \supset [T]$ be a minimal prime ideal. Then $\dim K\lbrace x\rbrace/P$ is finite.
\end{lemma}
\begin{proof}
	
	The proof is by induction on the minimal order of an element of $T$.  
	In the case that the order of the minimal element $u \in T $ is $-\infty$, then $u \in K$ and $[T]=K\lbrace x \rbrace$, which geometrically corresponds to the empty set. 
	
	Suppose now that every element of $T$ has nonnegative order and that the smallest element has order $r$. 
	We will suppose the lemma holds for all $T$ with an element order less than $r$. 
	Let $u \in T$ be an element of minimal order. 
	It suffices to show that for minimal $P\supset [u]$, $\dim K\lbrace x \rbrace/P < \infty$ (since all the other elements of $T$ just impose extra relations). 
	
	It suffices to factor $u$ into irreducible factors and then prove the lemma for each factor, so we suppose without loss of generality that $u$ is irreducible.  
	We know $[u] = P_0 \cap [u,s_u]$ where $P_0 = \sat_{s_u}([u])$ and $s_u$ is the separant of $u$ (c.f.~\cite[section 30, pg 49]{Kaplansky1976} and Theorem \ref{T:fundamentals-of-char-sets} (\ref{item:saturation-char-set})), and $P$ is either $P_0$ or a minimal prime over $[u,s_u]$ (by \cite[pg 49]{Kaplansky1976}).

  In the first case, since $s_u \not \in P_0$, $\dim K\lbrace x \rbrace/P_0 = \dim \left(K\lbrace x \rbrace/P_0\right) \left[\frac{1}{s_u}\right]$.
  By Theorem \ref{T:fundamentals-of-char-sets} (\ref{item:membership-char-set}) and Equation \eqref{E:ritt-division}, $\left(K\lbrace x \rbrace/P_0\right) \left[\frac{1}{s_u}\right]$ is generated by elements of order at most the order of $u$, and in particular is finite dimensional.
  In the second case, since the order of $s_u$ is less than the order of $u$, the lemma holds by induction.
\end{proof}

\section{The Dimension Conjecture}\label{S:dimension-conjecture}

The proof of the Jacobi Bound Conjecture relative to the Dimension Conjecture requires starts with a differential ideal $[u_1,\ldots,u_n] \subset K\lbrace x_1,\ldots,x_n\rbrace$ and then requires an estimate of the differential dimension of irreducible components over $[u_2,\ldots,u_n]$, after we lose a single equation. 
To make this estimate, we need to apply the Dimension Conjecture in the form of Conjecture~\ref{C:goosed}.
The actual application appears in Theorem~\ref{T:moving-dimension}.

There are several flavors of the Dimension Conjecture. 
We don't know of a reference that cleanly works out that all the different versions are equivalent so we do so here. 
\begin{itemize}
	\item The weakest-looking version appears in Cohn's paper \cite[pg 1]{Cohn1983}. 
	He attributes it to Lando. The conjecture asserts that if a system has fewer equations than variables, then the differential dimension must be positive. 
	This is Conjecture~\ref{C:cohn}.
	\item A slightly stronger-looking version appears in Ritt's book \cite[pg 178]{Ritt1950}. 
	This conjecture imposes a lower bound on the differential dimension based on the difference in the number of variables and equations. This is Conjecture~\ref{C:ritt}. 
	\item Finally, there is a version that is analogous to the statement of Krull's principal ideal theorem which is useful in inductive arguments. This is Conjecture~\ref{C:goosed}.
\end{itemize}
We now state the conjectures formally. 
\begin{conjecture}[Cohn's Dimension Conjecture {\cite[pg 1]{Cohn1983}}]\label{C:cohn}
	Consider $u_1,\ldots,u_i \in K\lbrace x_1,\ldots,x_n\rbrace$.  
	Suppose $[u_1,\ldots,u_i]$ is not the unit ideal. 
	If $P$ is a minimal prime above $ [u_1,\ldots,u_i]$ and $i<n$, then $\trdeg_K^{\partial}(\kappa(P))>0$. 
\end{conjecture}

\begin{conjecture}[Ritt's Dimension Conjecture{ \cite[pg 178]{Ritt1950}}]\label{C:ritt}
	Consider $u_1,\ldots,u_i \in K\lbrace x_1,\ldots,x_n\rbrace$.  
	Suppose $[u_1,\ldots,u_i]$ is not the unit ideal. 
	If $P$ is a minimal prime above $[u_1,\ldots,u_i]$ and $i<n$, then $\trdeg_K^{\partial}(\kappa(P))\geq n-i$. 
\end{conjecture}

\begin{conjecture}[PIT Dimension Conjecture]\label{C:goosed}
	Let $A$ be a differentially finitely generated differential $K$-domain. 
	Let $f \in A$ be a non-unit element. Let $P$ be a minimal prime ideal over $[f]$. 
	\begin{equation}\label{E:general-slicing}
		\trdeg^{\partial}_K(A/P) \geq \trdeg^{\partial}_K(A)-1. 
	\end{equation}
\end{conjecture}

\begin{proposition}
Conjectures \ref{C:cohn}, \ref{C:ritt}, and \ref{C:goosed} are equivalent.
\end{proposition}
\begin{proof}

$\eqref{C:ritt} \implies \eqref{C:cohn}$: Ritt's version gives lower bounds and Cohn's version does not. 
\medskip

$\eqref{C:cohn} \implies \eqref{C:ritt}$:
		Suppose Cohn's version.
		Suppose you have $m$ equations $u_1,\ldots,u_m$ in $K\lbrace x_1,\ldots,x_n\rbrace$ variables with $m<n$.
		Let $X=\lbrace x_1,\ldots,x_n\rbrace$.
		We will base change to $K(\lbrace Z\rbrace)$ where we partitioned $X=Y\amalg Z$ in such a way $u_i \notin K\lbrace Z \rbrace$. To do this we can think of ``protecting" some variable for each of our equations which we will not invert. 
		Since we have $m$ equations, there are $m$ forbidden variables (possibly with repeats). 
		Pick one variable that appears in the support for each of them.
		This leaves $n-m$ variables we may invert. 
		We will only invert $n-m-1$ variables. Call these variables the $Z$ variables. 
		Now we have $u_1,\ldots,u_m \in F\lbrace Y\rbrace$ where now where $ F=K(\lbrace Z\rbrace)$ and $\vert Y\vert =m+1$.
		By the Cohn's version of the dimension conjecture we have 
		$$ \trdeg_{F}^{\partial}( F\lbrace Y\rbrace/[u_1,\ldots,u_m]) \geq 1.$$
		Now, since differential transcendence degree is additive in towers, we have
		$$\trdeg^\partial_K(F\lbrace Y\rbrace/[u_1,\ldots,u_m])= \trdeg^{\partial}_{F}(F\lbrace Y\rbrace/[u_1,\ldots,u_m]) + \trdeg^{\partial}_K(F) \geq 1 + (m-n)-1 = m-n$$
  which proves the claim.
  The statement about differential transcendence degrees in towers of differential fields is \cite[pg 107, pg 112]{Kolchin1973}.
  When $A$ is $\partial$-finite type but not a field we use
  $$\trdeg_F^{\partial} (A) = \max\lbrace \trdeg_F^{\partial}\kappa(P)  : P \mbox{ minimal $\partial$-prime of $A$} \rbrace,$$ 
  where the result then follows since each $\kappa(P)$ is an extension of $F$, whence we can apply the theorem as stated in the reference.
\medskip
  
  $\eqref{C:goosed} \implies \eqref{C:ritt}$
		Suppose the PIT version.
		Let $P \supset [u_1,\ldots,u_i]$ be a prime ideal which is minimal with respect to this containment. 
		We will assume the PIT Dimension Conjecture and prove Cohn's Dimension Conjecture which is that $\trdeg_K^{\partial}(\kappa(P))>0$ if $i<n$.
		
		The proof is by induction on $i$.
		Let $A=K\lbrace x_1,\ldots,x_n\rbrace$.
		In the case $i=1$ we have $n\geq 2$.
		Since $u_1$ is a non-unit, the PIT Dimension Conjecture implies that $\trdeg^{\partial}_K(A/P) =\trdeg^{\partial}_K(A)-1=n-1 \geq n-1$. 
		
		Let $P_0$ be a minimal prime over $[u_1,\ldots,u_{i-1}]$. 
		Let $A=K\lbrace x_1,\ldots,x_n\rbrace/P_0$.
		By the inductive hypothesis we have that $\trdeg^{\partial}_K(A) \geq n-(i-1)$. 
		Consider $f=\bar{u}_i$ the image of $\bar{u}_i$ in $A$.
		Let $P$ be a minimal prime over $[f]$ in $A$.
		Since we assumed that $[u_1,\ldots,u_i]$ is not the unit ideal, $f$ is not a unit. 
		By the PIT Dimension Conjecture, we have that $\trdeg^{\partial}_K(A/P) \geq \trdeg^{\partial}_K(A) -1 \geq n-(i-1)-1=n-i$, which proves Ritt's version. 
  \medskip
  
  $\eqref{C:ritt} \implies \eqref{C:goosed}$:
		Suppose Ritt's version.
		Suppose  $A=K\lbrace x_1,\ldots,x_n\rbrace/P_0$ where $P_0$ is minimal over $[g_1,\ldots,g_h]$ and $\height^{\partial}(P_0)=h$ so that $\trdeg^{\partial}_K(A)=n-h$. 
		
		Now we will prove the PIT Dimension Conjecture.
		Let $\bar{f} \in A$ be a non-zero non-unit. 
		Let $\bar{P}$ be a minimal prime over $[\bar{f}]$. 
		Then $\bar{f}$ lifts to some $f \in K\lbrace x_1,\ldots,x_n\rbrace$ and $\bar{P}$ lifts to some prime $P$ where $\bar{P}=P/P_0$ which is minimal over $[g_1,\ldots,g_h,f]$. 
		Then applying Ritt's version gives $\dim^{\partial}(A/\bar{P}) = \dim^{\partial}(K\lbrace x_1,\ldots,x_n\rbrace/P) \geq  n-(h+1) = \dim^{\partial}(A)-1$.
\end{proof}

From now on we will refer to any of Conjecture~\ref{C:cohn}, Conjecture~\ref{C:ritt}, or Conjecture~\ref{C:goosed} as ``the Dimension Conjecture'' and conflate all of these forms.
For applications we will use Conjecture~\ref{C:goosed}.

\begin{proposition}\label{P:dimension-bounds}
	Assume the Dimension Conjecture (Conjecture~\ref{C:goosed}). 
	Let $P$ be a minimal prime over $[u_1,u_2,\ldots,u_n] \subset K\lbrace x_1,\ldots,x_n\rbrace$. Assume that $P$ has differential dimension zero. 
	Let $P' \supset [u_2,\ldots,u_n]$ be a prime ideal which is contained in $P$ and is minimal over $[u_2,\ldots,u_n]$ with respect to this property.
	We claim that this family has differential dimension at most one, i.e.,
	$$\trdeg^{\partial}(\kappa(P'))\leq 1,$$
	and that it has differential dimension exactly $1$ when $u_1$ is not in $P'$ and not a unit modulo $P'$.
\end{proposition}
\begin{proof}
	As in the setup of the Dimension Conjecture we let $R=K\lbrace x_1,\ldots,x_n\rbrace/P'$. Let $f=\overline{u}_1$ be the image of $u_1$ in $P'$, which we suppose is non-zero, and let $Q=P$ be the minimal prime above it. 
	
	From the Dimension Conjecture (Conjecture~\ref{C:goosed}). $$\trdeg^{\partial}(\kappa(P))\geq\trdeg^{\partial}(\kappa(P'))-1.$$
	But $\trdeg^{\partial}_K(\kappa(P))=0$ which implies $\trdeg^{\partial}(\kappa(P'))\leq 1$. 
\end{proof}

\begin{proposition}\label{P:behavior-of-intermediates}
	Assume the Dimension Conjecture (Conjecture~\ref{C:goosed}). 
	Let $P$ be a minimal prime over $[u_1,\ldots,u_n] \subset K\lbrace x_1,\ldots,x_n\rbrace$. 
	The following are equivalent:
	\begin{enumerate}
		\item $\trdeg^{\partial}(\kappa(P))=0$.
		\item For each $i$ there exists a prime ideal $P_{i-1} \subset P$ such that
  \begin{enumerate}
  \item $[u_i,\ldots,u_n] \subset P_{i-1}$,
  \item $\trdeg^{\partial}_K(\kappa(P_i)) = i$, and 
  \item $u_i$ in $K\lbrace x_1,\ldots,x_n\rbrace/P_i$ is non-zero and not a unit, 
  \end{enumerate}
  and such that $P_{i-1}$ is minimal with respect to these properties.
	\end{enumerate}
\end{proposition}
\begin{proof}
	We give some setup:  consider now a chain of primes $0=P_n \subset P_{n-1} \subset \cdots \subset P_0 $ where $P_0=P$, $P_{1}$ is minimal over $[u_2,\ldots,u_n]$ which is contained in $P_0$, and for $i>0$ we proceed inductively saying that $P_{i}$ is a minimal prime over $[u_{i},\ldots,u_n]$ which is contained in $P_{i-1}$ etc.
	Here is a diagram:
	$$
	\begin{tikzcd}
	P_{n-1} \arrow[r, hook] & P_{n-2} \arrow[r, hook] & \cdots \arrow[r, hook] & P_1 \arrow[r, hook] & P_0=P  \\[4pt]
	{[u_n]} \arrow[u, hook] \arrow[r, hook]
	& {[u_{n-1},u_n]} \arrow[u, hook] \arrow[r, hook]
	& \cdots \arrow[r, hook]
	& {[u_2,\ldots,u_n]} \arrow[u, hook] \arrow[r, hook]
	& {[u_1,\ldots,u_n]} \arrow[u, hook].
	\end{tikzcd}
	$$
	
	We will show that if $\trdeg^{\partial}(\kappa(P))=0$, then all the intermediate transcendence degrees must behave predictably, as in the statement of this proposition.
	The differential transcendence degree from $\kappa(P_n)=K(\lbrace x_1,\ldots,x_n\rbrace)$ to $\kappa(P_0)$ starts at $n$ and drops to zero. At each stage it can only drop by at most one by the dimension conjecture. 
	Indeed, since $P_i$ is minimal over $P_{i-1}+[u_1,\ldots,u_i]$, we know that $\bar{P}_i \subset \bar{R}_{i-1} = R/P_{i-1}$ is minimal over $\bar{u}_i \in R/P_{i-1}$ where $\bar{u}_i$ is the image of $u_i$ in $R/P_{i-1}$.
	This implies that $\bar{P}_i$ has height at most one, and by the dimension conjecture $\trdeg^{\partial}_K(R/P_{i}) \geq \trdeg_K^{\partial}(R/P_{i-1})-1$.
	It follows that the differential transcendence degree drops at every possible step by exactly one and that $\trdeg^{\partial}(\kappa(P_i))=i$. 
	
	Conversely, we have precisely the non-degeneracy hypotheses that tell us that at each stage the differential transcendence degree drops by one. 		
\end{proof}

\section{Ritt's Cycle Trick}
The following elementary Lemma is useful for fixing our notation (and for sanity checks).
Many of the following sections involve delicate manipulations of matrices and transversals.
The Lemma essentially says that row and column operations are left and right matrix multiplication by elementary matrices, and spells out how rearrangements of matrices modify transversals. 
We trust that many readers could work this out on their own, but we have included it for convenience. 
\begin{lemma}\label{L:permutations-on-transversals}
	Suppose that $A=(a_{i,j})$ is an order matrix for a system $u_1=0,\ldots,u_n=0$ with variables $x_1,\ldots,x_n$. Suppose that 
	$\tdet(A) = \sum_{i=1}^n a_{i,\rho(i)}.$
	Then, if we permute the equations by $\sigma \in S_n$ and we permute the variables by $\tau \in S_n$ we get a new system $v_1,\ldots,v_n$ in variables $y_1,\ldots,y_n$ with order matrix $B=(b_{i,j})$ and maximum transversal $ \tdet(B) = \sum_{i=1}^n b_{i,\rho'(i)} $
	where 
	\begin{equation}\label{eqn:right-transformation}
		\rho' = \tau^{-1}\rho \sigma. 
	\end{equation}
\end{lemma}
In the lemma, $\tau$ corresponds to row operations and $\sigma$ corresponds to column operations.
  
\begin{proof}
	Suppose $\rho \in S_n$ defines the transversal for an order matrix $A$ so that 
	$$ \tdet(A) = \sum_{i=1}^n a_{i, \rho(i)}=\sum_{j=1}^n a_{\rho^{-1}(j),j}. $$
	Let $B$ be the matrix obtained by permuting the equations so that $b_{i,j} = a_{\sigma(i),j}$.
  Then $a_{i,j} = b_{\sigma^{-1}(i),j}$ and
  $$\tdet(B) = \tdet(A)=\sum_{j=1}^n a_{\rho^{-1}(j),j}= \sum_{i=1}^n b_{\sigma^{-1}(\rho^{-1}(j)),j}$$
  so that $\sigma^{-1}\rho^{-1} = (\rho\sigma)^{-1}$ and $\rho \mapsto \rho \sigma$.
	
	Similarly, if we permute the variables so that $b_{i,j} = a_{i,\tau(j)}$ and $b_{i,\tau^{-1}(j)} = a_{i,j}$, then
  $$\tdet(B) = \tdet(A) = \sum_{i=1}^n a_{i,\rho(i)}= \sum_{i=1}^n b_{i,\tau^{-1}(\rho(i))}$$
  and we have that $\rho \mapsto \tau^{-1}\rho$. 
\end{proof}

\begin{remark}
	The above lemma defines a right action of permutations on the diagonal of the order matrix. If we use permutations of the rows to define the transversal of the tropical determinants we still get a contravariant action. To see this, suppose $\tdet(A) = \sum_{j=1}^n a_{\lambda(j),j}.$
	In this case $\lambda = \rho^{-1}$ and then $\lambda' = (\rho')^{-1} = (\tau^{-1} \rho \sigma)^{-1} = \sigma^{-1} \rho^{-1} \tau = \sigma^{-1} \lambda \tau$ so that 
	\begin{equation}\label{eqn:left-transformation}
		\lambda' = \sigma^{-1} \lambda \tau.
	\end{equation}
\end{remark}

The following definition will greatly facilitate our manipulations of transversals.

\begin{definition}
  Let $A=(a_{i,j})$ be an $n\times n$ matrix with entries in $\ZZ_{\geq 0} \cup \lbrace -\infty \rbrace$.
  Let $\tau = (i_1i_2\ldots i_s) \in S_n$ be a cycle.
  Then we define
	\begin{equation}\label{E:cyclic-sum}
	a_{\tau} = a_{i_1,i_2}+a_{i_2,i_3} +\cdots + a_{i_s,i_1}.
	\end{equation}
  \end{definition}

  \begin{remark}
Every tropical determinant can be written as a sum of such cyclic factors. 
Indeed, if $\tdet(A) = \sum_{i=1}^n a_{i,\rho(i)}$ for $\rho \in S_n$ and $\rho=\tau_1\tau_2\ldots\tau_c$ is its product of disjoint cycles, then 
\begin{equation}\label{E:cyclic-sum-decomposition}
\tdet(A) = \sum_{j=1}^c a_{\tau_j},
\end{equation}
where $a_{\tau_j}$ is the cyclic-sum as defined in \eqref{E:cyclic-sum}.

\begin{definition}
We will call \eqref{E:cyclic-sum-decomposition} a \defi{cyclic sum decomposition} of the transversal.
\end{definition}

Our convention is that the symmetric group acts on the left, so that we read permutations from right to left (i.e., $(12)(13) = (132)$).
\end{remark}
  
\begin{example}\label{E:disjoint-cycles}
	Suppose that we have $A=(a_{i,j})$ which is a $5\times 5$ matrix with entries in $\ZZ_{\geq 0} \cup \lbrace -\infty \rbrace$. 
	Suppose that $\tdet(A) = \sum_{i=1}^n a_{i,\rho(i)}$ where $\rho \in S_5$ written as a product of disjoint cycles is  $\rho =(13)(245)$.
	Then we have 
	 $$ \tdet(A) = a_{1,3} + a_{2,4} + a_{3,1} + a_{4,5} + a_{5,2}.$$
	We can regather this tropical determinant into factors corresponding to its cycles $(13)$ and $(245)$. 
	$$ \tdet(A) = (a_{1,3} + a_{3,1}) + (a_{2,4}+a_{4,5}+a_{5,2}) = a_{(13)} + a_{(245)}. $$

\end{example}

In our reductions we make extensive use of \defi{Ritt's cycle trick}, which is an efficient way to formulate the typical rearrangements one makes when working with transversals. 
\begin{lemma}[Ritt's Cycle Trick]\label{lem:ritt-permutation}
	Let $A=(a_{i,j})$ be an $n\times n$ matrix with entries in $\ZZ_{\geq 0} \cup \lbrace -\infty \rbrace$. 
	Suppose that $\tdet(A) = \sum_{i=1}^n a_{i,i}.$
	If $\tau \in S_n$ is a cycle, then 
  \[
  a_{\tau} \leq \sum_{i \in \supp(\tau)}a_{i,i}.
  \]
\end{lemma}
\begin{proof}
	Since the transversal $A_{\tau}$ is less than or equal to $\tdet(A)$, we have that 
	$$a_{\tau}+\sum_{i \notin \supp(\tau)} a_{i,i} = A_{\tau} \leq \tdet(A) = \sum_{i \in \supp(\tau)}a_{i,i} +  \sum_{i \notin\supp(\tau)} a_{i,i}.$$
	Cancelling the terms not in the support of $\tau$ gives the result.  
\end{proof}

\section{Ritt's Reductions: Divisions That Preserve Jacobi numbers}
\label{section:ritts-reductions}

We now introduce Ritt's first and second reductions (Proposition~\ref{P:ritt-first} and Proposition~\ref{P:ritt-second} respectively). 
These are the key steps in Ritt's proof of the linear case of the Jacobi Bound Conjecture  \cite[pg 310]{Ritt1935} and were discovered before the theory of characteristic sets were developed. Reflecting on Ritt's proofs, we realized that his claims about Jacobi numbers still hold for nonlinear systems. Since this observation is crucial to the proof of our main theorem (see Sections \ref {S:ritt-pencil} and  \ref{section:Dimension-Conjecture-Implies-Jacobi-Bound-Conjecture}), we exposit his proof in the next few sections, and crucially fix several minor gaps and one major gap (Section \ref{S:filling-ritts-gap}).
\\

We now give the hypotheses of Ritt's first reduction (Proposition~\ref{P:ritt-first}) and Ritt's second reduction (Proposition~\ref{P:ritt-second}) names. 
We will call them Ritt's first form and Ritt's second form respectively. 
\begin{definition}\label{D:ritt-forms}
	Let $R=K\lbrace x_1,\ldots,x_n\rbrace$.
	Let $I = [u_1,u_2, \ldots,u_n] \subset R$, with order matrix 
	$A=(a_{i,j})$, where $a_{i,j} = \ord^{\partial}_{x_j}(u_i)$.
  We say that the system is in \defi{Ritt's first form} if
	\begin{equation}\label{E:first-form}
		\tdet(A) = \sum_{i=1}^n a_{i,i} \quad  \text{ and } \quad a_{2,1}\geq a_{1,1} \neq -\infty, 
	\end{equation}  
and in \defi{Ritt's second form} if
	\begin{eqnarray}
  \tdet(A) = a_{1,n} + \left(\sum_{i=2}^{n-1}  a_{i,i}\right) + a_{n,1}, \label{E:second-form-1}\\
  \tdet(\widetilde{A}_{n,n})=\sum_{i=1}^{n-1}a_{i,i} \neq -\infty, \text{ and } \label{E:second-form-2}\\
	a_{n,1}=\max \lbrace  a_{1,1}, a_{2,1}, \ldots, a_{n,1}\rbrace\label{E:second-form-3}.  
  \end{eqnarray}
We note that $\ref{E:second-form-2}$ and $\ref{E:second-form-3}$ imply that $a_{1,1}, a_{n,1} \neq -\infty$.  
  \end{definition}

If a system is not in Ritt's first or second form, we report that extensive experimentation in Magma shows that Ritt any division is likely to increase the strong Jacobi number;
see Example~\ref{E:J-increasing-second-form}. In particular, as far as we can tell, Ritt identified \emph{precisely} the two situations in which Ritt division does not increase the Jacobi number of a system.

\begin{example}\label{E:J-increasing-second-form}
	Consider the system of differential equations 
	$$\begin{cases}
		u_1=x + x'+y^{(2)}+z^{(3)},\\
		u_2=x'+y'+z', \\
		u_3=x^{(2)}+y'+z'.
	\end{cases}$$
	This has order matrix
$$\begin{pmatrix}
1 & 2 & \boxed{3}\\
1 & \boxed{1} & 1\\
\boxed{2} & 1 & 1\\
\end{pmatrix} $$
and Jacobi number 6; the boxed terms form a maximal transversal. This system satisfies
\eqref{E:second-form-2} and the first part of \eqref{E:second-form-1} but not the second part of \eqref{E:second-form-2}.
	Ritt dividing the third equation by the first gives the system 
	$$\begin{cases}
		u_1=x + x'+y^{(2)}+z^{(3)},\\
		u_2=x'+y'+z', \\
		v_3=x'+y' - y^{(3)} +z' - z^{(4)}.
  \end{cases}$$
  which has order matrix
  $$\begin{pmatrix}
1 & \boxed{2} & 3\\
\boxed{1} & 1 & 1\\
1 & 3 & \boxed{4}\\
\end{pmatrix} $$
and  Jacobi number 7.
  
\end{example}

\begin{proposition}[Reduction from Ritt's first form]\label{P:ritt-first}
	Let $R=K\lbrace x_1,\ldots,x_n\rbrace$.
	Let $I = [u_1,u_2, \ldots,u_n] \subset R$, with order matrix 
	$A=(a_{i,j})$, where $a_{i,j} = \ord^{\partial}_{x_j}(u_i)$.
	Suppose that this system is in Ritt's first form (Definition \ref{D:ritt-forms}). 
	Let $h = s u_2 - Q u_1$ be the Ritt remainder after $\partial$-division of $u_2$ by $u_1$ in the variable $x_1$ (so only separants are used and $s$ is in the multiplicative set generated by $s_{1,1}$ and $Q\in R[\partial]$).
  
	Then $J(u_1,h,u_3,\ldots,u_n) \leq J(u_1,u_2,u_3,\ldots,u_n)$. 
\end{proposition}
\begin{proof}
	
	Suppose that $J(u_1,h,u_3,\ldots,u_n) > J(u_1,u_2,\ldots,u_n)$.
	Let $B=(b_{i,j})$ be the order matrix for $(u_1,h,u_3,\ldots,u_n).$ 
	Let  $\rho \in S_n$ be a permutation such that 
	$$J(u_1,h,u_3,\ldots,u_n) = \tdet(B) = \sum_{i=1}^n b_{i,\rho(i)}.$$
	Note that because we are only modifying the second equation we have $b_{i,j}=a_{i,j}$ for $i\neq 2$.
	Also, for $i=2$ we always have 
	\begin{equation}\label{E:division-inequality}
	b_{2,j}\leq \max \lbrace a_{2,j},a_{1,j}+a_{2,1}-a_{1,1} \rbrace,
	\end{equation}
	since we need to take $a_{2,1}-a_{1,1}$ many derivatives in the Ritt division. 
	We may suppose that there exists some $j$ such that $a_{2,j}<b_{2,j} \leq a_{1,j}+a_{2,1}-a_{1,1}$ (otherwise the proof is trivial) and that $\rho(2)=j$ with that particular $j$ (otherwise $b_{2,j}$ never appears and we don't need to worry about it).
	
	We now write the permutation $\rho \in S_n$ as a product of cycles and make arguments based on the support of the permutation. 
	We recall the support of a permutation is the set of elements in $\lbrace 1,2, \ldots, n \rbrace$ that are not fixed by $\rho$ and we denote the support of $\rho$ by $\supp(\rho)$. 
	
	Also, for a permutation $\tau \in S_n$ we will let $A_{\tau} = \sum_{i=1}^n a_{i,\tau(i)}$ denote the transversal associated to $\tau$. 
	By hypothesis we always have $A_{\tau} \leq J(u_1,\ldots,u_n)=\sum_{i=1}^n a_{i,i}$.

	In the first case we suppose $1\notin \supp(\rho)$, i.e.~$\rho(1)=1$. 
	In this case we get some cancellation:
	\begin{align*}
	\tdet(B) &= \sum_{i=1}^n b_{i,\rho(i)} \\
	&= a_{1,\rho(1)} + (a_{1,\rho(2)} + a_{2,1} - a_{1,1}) + a_{3,\rho(3)} + \cdots + a_{n,\rho(n)}\\
	&=a_{1,\rho(2)} + a_{2,1}  + a_{3,\rho(3)} + \cdots + a_{n,\rho(n)}\\
	&=a_{1,\tau(1)} + a_{2,\tau(2)}  + a_{3,\tau(3)} + \cdots + a_{n,\tau(n)} =A_{\tau} \leq \tdet(A)
	\end{align*}
	where $\tau=\rho \rho'$ where $\rho' = (1\rho(2))$ is the transposition that interchanges $\rho(2)$ and $1$. 
	
	In the second case we suppose $1 \in \supp(\rho)$ and $1$ is not in the orbit of 2 under $\rho$. 
	In this case we can apply Ritt's cycle trick. 
	Write $\rho=\tau_1\tau_2\cdots \tau_s$ where $\tau_i$ are disjoint cycles. 
	Without loss of generality we will assume $1 \in \supp(\tau_1)$. 
	Using equation \eqref{E:cyclic-sum-decomposition} we have  
	\begin{align*}
	\tdet(B)&= \sum_{i=1}^n b_{i,\rho(i)} = \sum_{j=1}^s b_{\tau_j}.
	\end{align*}
	We then combine the cycle inequality for $\tau_1$ (Lemma~\ref{lem:ritt-permutation}) and the division inequality (equation \eqref{E:division-inequality}) to get a similar cancellation of the $a_{1,1}$ terms:
	\begin{align*}
	\tdet(B) = b_{\tau_1}+\sum_{j=2}^s  b_{\tau_j} & \leq \left(\sum_{j \in \supp(\tau_1)} a_{j,j}\right) + \left(a_{1,\rho(2)} + a_{2,1} -a_{1,1} \right) + \left(\sum_{j\notin \supp(\tau_1), j\neq 2 } a_{j,\rho(j)}\right)\\
	&=\left(\sum_{j \in \supp(\tau_1), j\neq 1} a_{j,j}\right) + \left(a_{1,\rho(2)} + a_{2,1}\right) +\left( \sum_{j \notin \supp(\tau_1), j \neq 2} a_{j,\rho(j)}\right)=A_{\sigma}.
	\end{align*}
	Where $\sigma=\tau_1^{-1}\cdot 	\rho \cdot  (1\rho(2))$ so that $\sigma(1)=\rho(2)$ and $\sigma(2)=1$.

	In the third case we suppose $1 \in \supp(\rho)$ and that $1$ and $2$ are in the same orbit under $\rho$.\footnote{Ritt omitted this case in his original paper \cite{Ritt1935}.
	On page 308 equation 19, there seems to be an implicit assumption that $2$ isn't in the orbit of $1$ under the given permutation.
	}
We write $\rho = \tau_1\tau_2\cdots \tau_s$ as a cycle decomposition.
We will suppose that $1,2 \in \supp(\tau_1)$.
Then $\rho^i(1)=2$ for some $i$ and $\rho^j(2)=1$ for some $j$.  
Then
\begin{align*}
  b_{\tau_1} =&\,b_{1,\rho(1)} + b_{\rho(1),\rho^2(1)} + \cdots + b_{\rho^{i-1}(1),2} + b_{2,\rho(2)} + b_{\rho(2),\rho^2(2)} + \cdots + b_{\rho^i(2),1} \\
  \leq & \,a_{1,\rho(1)} + a_{\rho(1),\rho^2(1)} + \cdots + a_{\rho^{i-1}(1),2} + (a_{1,\rho(2)}+a_{2,1}-a_{1,1}) + a_{\rho(2),\rho^2(2)} + \cdots + a_{\rho^{j-1}(2),1} \\
  =&\, (a_{1,\rho(1)} + a_{\rho(1),\rho^2(1)} + \cdots + a_{\rho^{i-1}(1),2} + a_{2,1}) +(a_{1,\rho(2)}+ a_{\rho(2),\rho^2(2)} + \cdots + a_{\rho^{j-1}(2),1})-a_{1,1}\\
  =&\, a_{\sigma_1} + a_{\sigma_2}-a_{1,1}
\end{align*}
where $\sigma_1 = (1\rho(1)\rho^2(1)\cdots \rho^{i-1}(1)2)$ and $\sigma_2 = (\rho(2)\rho^2(2)\cdots \rho^{j-2}(2)1)$,
and where the second line follows from $b_{2,\rho(2)} \leq a_{1,\rho(2)}+a_{2,1}-a_{1,1}$.
Applying Lemma~\ref{lem:ritt-permutation} gives
\[
  a_{\sigma_1} + a_{\sigma_2}-a_{1,1}  
  \leq \, \left(\sum_{i \in \supp(\sigma_1)} a_{i,i}\right) + \left(\sum_{i \in \supp(\sigma_2)} a_{i,i}\right) - a_{1,1}.
  \]
Since $\supp \sigma_1 \cup \supp \sigma_2 = \supp \tau_1$ and $\supp \sigma_1 \cap \supp \sigma_2 = \{1\}$ (i.e., $a_{1,1}$ appears once in each sum), this last expression  is equal to $\sum_{i \in \supp(\tau_1)} a_{i,i}$. We conclude that  $b_{\tau_1} \leq  \sum_{i \in \supp(\tau_1)} a_{i,i}$, and in particular
\begin{align*}
\tdet(B)&= \sum_{j=1}^s b_{\tau_j}=b_{\tau_1}+\sum_{j=2}^s a_{\tau_j} \leq \sum_{i \in \supp(\tau_1)} a_{i,i}+\sum_{j=2}^s a_{\tau_j} = A_{\tau_1^{-1}\rho} \leq \tdet(A).
\end{align*}
\end{proof}

\begin{proposition}[Reduction From Ritt's Second Form]\label{P:ritt-second}
	Let $R=K\lbrace x_1,\ldots,x_n\rbrace$.
	Let $I = [u_1,u_2, \ldots,u_n] \subset R$, with order matrix
	$A=(a_{i,j})$, where $a_{i,j} = \ord^{\partial}_{x_j}(u_i)$.
	Suppose that the system is in Ritt's second form (Definition \ref{D:ritt-forms}). 	
	Let $h = s u_n - Q u_1$ be the Ritt remainder after $\partial$-division of $u_n$ by $u_1$ in the variable $x_1$ (so only separants are used and $s$ is in the multiplicative set generated by $s_{1,1}$ and $Q\in R[\partial]$).

	Then $J(u_1,u_2,\ldots,u_{n-1},h) \leq J(u_1,u_2,\ldots,u_{n-1},u_n)$. 
\end{proposition}

\begin{remark}
\label{remark:cycle-trick-how-to-use-hypothesis}  
  One difference between this and Proposition~\ref{P:ritt-first}  is that the maximal transversal of $A$ is not assumed to be along the diagonal.
  Thus Ritt's Cycle Trick (Lemma~\ref{lem:ritt-permutation}) does not apply to $a_{\tau}$ for all $\tau$.
  But the assumption that $\tdet(\widetilde{A}_{n,n}) = \sum_{i=1}^{n-1} a_{i,i}$ implies that for any $\tau$ such that $n \not \in \supp \tau$, applying Lemma~\ref{lem:ritt-permutation} to the submatrix $\widetilde{A}_{n,n}$ still gives
\[
a_{\tau} \leq \sum_{i\in \supp(\tau)} a_{i,i}.
\]
\end{remark}

\begin{proof}[Proof of Proposition~\ref{P:ritt-second}]

		Since $a_{n,1}$ is the maximum of its column, $a_{n,1}\neq-\infty$;
  indeed, otherwise the entire column would consist of $-\infty$'s, i.e.~the variable $x_1$ and it's derivatives would be absent from the equations, and the components would therefore all be infinite dimensional.

		Let $B=(b_{i,j})$ be the order matrix for the new system. 
		We have that $b_{i,j}=a_{i,j}$ for $i\neq n$ and $b_{n,j} \leq \max\lbrace a_{n,j}, a_{1,j}+a_{n,1}-a_{1,1} \rbrace$.
  If $b_{n,j} = a_{n,j}$ for all $j$, then $A = B$ and there is nothing to prove.
  Similarly, if $b_{n,j} = a_{n,j}$ for $j$ such that $b_{n,j}$ appears as part of a maximal transversal, then $\tdet(A) = \tdet(B)$ and there is nothing to prove.

  Thus we assume that for some $j$, $b_{n,j}$ appears as part of a maximal transversal, and $b_{n,j} > a_{n,j}$.
  In this case, we still have that $b_{n,j} \leq a_{1,j}+a_{n,1}-a_{1,1}$.
		Let $\rho$ be a permutation such that $\tdet(B)=\sum_{i=1}^n b_{i,\rho(i)}$ so that 
		 $$ b_{n,\rho(n)} \leq a_{1,\rho(n)}+a_{n,1}-a_{1,1}.$$
		We then have that
		\begin{align*}
		\tdet(B)=\sum_{i=1}^n b_{i,\rho(i)} & \leq \left(\sum_{i=1}^{n-1} a_{i,\rho(i)}\right) + (a_{1,\rho(n)}+a_{n,1}-a_{1,1}).
		\end{align*} 
		We break this down into cases similar to Proposition~\ref{P:ritt-first}.
		
		Suppose $\rho(1)=1$. Then $a_{1,\rho(1)}$ and $a_{1,1}$ cancel giving
		\begin{align*}
		a_{1,\rho(1)} + a_{2,\rho(2)} + \cdots + a_{n-1,\rho(n-1)} + (a_{1,\rho(n)}+a_{n,1}-a_{1,1}) &=\\
  a_{1,\rho(n)}+a_{2,\rho(2)} + \cdots + a_{n-1,\rho(n-1)} + a_{n,1} &=  A_{\tau}
		\end{align*} 
		where $\tau=\rho\sigma$ where $\sigma = (1n)$. 
		This implies $\tdet(B)\leq \tdet(A)$.

		Next, write $\rho$ as a product of disjoint cycles so that 
		$$\rho=\tau_1\tau_2\cdots\tau_s.$$
		We will suppose that $1 \in \supp(\tau_1)$ (we are now always in a case where $\rho(1)\neq 1$).
		
		First suppose  $n$ is not in the orbit of $1$ under $\rho$  (i.e.~$n\notin \supp(\tau_1))$. 
		We will suppose that $n \in \supp(\tau_s)$.
		Since $n \not \in \supp \tau_1$, by Remark~\ref{remark:cycle-trick-how-to-use-hypothesis}, $a_{\tau_1}\leq \sum_{i\in \supp(\tau_1)} a_{i,i}$. 
		This implies 
		\begin{align*}
		b_{\tau_1} + \left(\sum_{j=2}^{s-1} b_{\tau_j} \right) + b_{\tau_s} 
		=&\,a_{\tau_1} + \left(\sum_{j=2}^{s-1} a_{\tau_j}\right) + b_{\tau_s}\\
		\leq &\,a_{\tau_1} + \left(\sum_{j=2}^{s-1} a_{\tau_j}\right) + \left(b_{\tau_s}-b_{n,\rho\left(n\right)}\right)+\left(a_{1,\rho\left(n\right)}+a_{n,1}-a_{1,1}\right)\\
		\leq &\, \left(\left(\sum_{i \in \supp\left(\tau_1\right)} a_{i,i}\right)-a_{1,1}\right) + \left(\sum_{j=2}^{s-1} a_{\tau_j}\right) + \left(\left(\sum_{i \in \supp\left(\tau_s\right)} b_{i,\rho\left(i\right)}\right)-b_{n,\rho\left(n\right)}\right) + a_{1,\rho\left(n\right)} +a_{n,1} \\
  =&\,\left(\sum_{i \in \supp\left(\tau_1\right), i \neq 1} a_{i,i}\right) + \left(\sum_{j=2}^{s-1} a_{\tau_j}\right) + a_{1,\rho\left(n\right)}+
  \left(\sum_{i \in \supp\left(\tau_s\right), i \neq n} a_{i,\rho\left(i\right)}\right) + a_{n,1}\\
  =&\,\left(\sum_{i \in \supp\left(\tau_1\right), i \neq 1} a_{i,i}\right) + \left(\sum_{j=2}^{s-1} a_{\tau_j}\right) + a_{\tau_s(1n)} \\
		=&\, A_{\tau}  
		\end{align*} 
		where $\tau=\tau_2\cdots\tau_s(1n)$.
  (Note that $\tau(1) =\rho(n)$ and $\tau(n)=1$.)
		By Lemma~\ref{lem:ritt-permutation}, $A_{\tau} \leq \tdet(A)$; since $\tdet(B) \leq A_{\tau}$, we conclude that $\tdet(B) \leq \tdet(A)$. 
		
	Suppose now that both $1$ and $n$ are in the support of $\tau_1$. 
	We then have 
	 $$ \tdet(B) = b_{\tau_1} + b_{\tau_2} + \cdots + b_{\tau_s} = b_{\tau_1} + a_{\tau_2} + \cdots + a_{\tau_s},$$
	and it remains to bound $b_{\tau_1}$ using $b_{n,\rho(n)} = a_{1,\rho(n)}+a_{n,1}-a_{1,1}$. 
	We divide up $b_{\tau_1}$ knowing that there exists some $i$ where $\rho^i(1)=n$ and some $j$ such that $\rho^j(n)=1$. Then
	\begin{align*}
	b_{\tau_1} =&(b_{1,\rho(1)} + b_{\rho(1),\rho^2(1)} + \cdots + b_{\rho^{i-1}(1),n})+ b_{n,\rho(n)} + (b_{\rho(n),\rho^2(n)} + \cdots + b_{\rho^{j-1}(n),1}) \\
	=& (a_{1,\rho(1)} + a_{\rho(1),\rho^2(1)} + \cdots + a_{\rho^{i-1}(1),n}) + (a_{1,\rho(n)}+a_{n,1}-a_{1,1}) + (a_{\rho(n),\rho^2(n)} + \cdots + a_{\rho^{j-1}(n),1}) \\
	=& (a_{1,\rho(1)} + a_{\rho(1),\rho^2(1)} + \cdots + a_{\rho^{i-1}(1),n} + a_{n,1}) +(a_{1,\rho(n)}+ a_{\rho(n),\rho^2(n)} + \cdots + a_{\rho^{j-1}(n),1})-a_{1,1}.
	\end{align*}
	If we let $\sigma_1=(1\rho(1)\rho^2(1)\cdots \rho^{i-1}(1)n)$ and $\sigma_2=(1\rho(n)\rho^2(n)\cdots \rho^{j-1}(n))$, then 
	 $$ b_{\tau_1} = a_{\sigma_1} + a_{\sigma_2} - a_{1,1}.$$
	Since $n \not \in \supp \sigma_2$, by Remark~\ref{remark:cycle-trick-how-to-use-hypothesis} we get 
	$$ b_{\tau_1} \leq a_{\sigma_1} + \sum_{i \in \supp(\sigma_2), i \neq 1} a_{i,i}.$$
	This then gives 
  $$ \tdet(B) = b_{\tau_1} + a_{\tau_2} + \cdots + a_{\tau_s} \leq a_{\sigma_1} + \left(\sum_{i \in \supp(\sigma_2), i \neq 1} a_{i,i}\right) + a_{\tau_2} + \cdots + a_{\tau_s} = A_{ \sigma_1\tau_1^{-1}\rho} \leq \tdet(A),$$
  completing the proof.
	\end{proof}

\section{Existence of the First Form}

\begin{proposition}[\cite{Ritt1935}]\label{P:first-form}
  Let $A$ be an $n\times n$ with entries in $\ZZ_{\geq 0} \cup \lbrace -\infty \rbrace$.
  Suppose that for every transversal $\rho$ of $A$, $a_{1,\rho(1)}\neq -\infty$ and $a_{1,\rho(1)}$ is not the maximum of the first column. Suppose additionally that at least one other entry of the first column is not equal to $-\infty$.  
  Then $A$ can be brought to Ritt's first form (Definition \ref{D:ritt-forms}) using row swaps and column swaps which do not involve the first column.
\end{proposition}

\begin{proof}  
  Suppose that $\tdet(A) = \sum_{i=1}^n a_{i,\rho(i)}$. Taking $\tau = \rho$ in Lemma \ref{L:permutations-on-transversals} (i.e., rearranging the rows of $A$) gives a matrix $B$ with $\tdet(B) = \sum_{i=1}^n b_{i,i}$. We have not rearranged any columns, so by assumption, $b_{1,1} \leq \max \lbrace b_{1,1},b_{2,1},\ldots,b_{n,1}\rbrace$, and there exists some $i$ such that $b_{1,1} \leq b_{i,1}$. Applying Lemma \ref{L:permutations-on-transversals} again with $\tau = \sigma = (2i)$ gives a matrix $C$ such that $\tdet(C) = \sum_{i=1}^n c_{i,i}$ and $c_{1,1} \leq c_{2,1}$, i.e.~$C$ is in Ritt's first form.
\end{proof}

\section{Filling Ritt's Gap: Existence Of The Second Form}
\label{S:filling-ritts-gap}

In preparing this article the authors noticed a nontrivial gap in Ritt's proof of the linear case of the Jacobi Bound Conjecture in \cite{Ritt1935}. 
The gap is in his proof that any $n\times n$ matrix with entries in $\ZZ_{\geq 0} \cup \lbrace -\infty \rbrace$ satisfying certain hypotheses can be put into Ritt's second form (Definition \ref{D:ritt-forms}).
Ritt claims at the end of his proof that this is evident; we found that in order to prove such a result we had to invoke Hall's Matching Theorem.

We will state the gap now as a proposition and then give a proof for it at the end of this section once we have set up the appropriate combinatorial machinery.

\begin{proposition}[{Gap of \cite{Ritt1935}}]\label{P:second-form}
  Let $A$ be an $n\times n$ with entries in $\ZZ_{\geq 0} \cup \lbrace -\infty \rbrace$.
  Suppose that for every transversal $\rho$ of $A$, $a_{1,\rho(1)}\neq -\infty$ and $a_{1,\rho(1)}$ is the maximum of the first column. Suppose additionally that at least one other entry of the first column is not equal to $-\infty$.
  Then $A$ can be brought to Ritt's first form (Definition \ref{D:ritt-forms}) using row swaps and column swaps which do not involve the first column or the last row. 
\end{proposition}

Our approach is a delicate proof by contradiction that relies on Hall's Matching Theorem.
\begin{definition}
	For a bipartite multigraph $G=(V,E)$ with $V=X \amalg Y$, a \defi{saturated matching} is a subset of edges of size $\vert X\vert$ that pairs each element of $X$ with an element of $Y$.  
\end{definition}
Another way of describing a matching is as a $1$-regular bipartite subgraph. 

In what follows, for $S$ a subset of $X$ we will let $N(S)=\{y\in Y : \exists\, x\in S \text{ with } (x,y)\in E\}$ denote the \defi{neighborhood} of $S$ in $Y$

\begin{theorem}[Hall's Matching Theorem {\cite[pg 110, Theorem 3.1.11]{West1996}}]
Let $G=(V,E)$ be a bipartite graph with partition $V=X\amalg Y$.
Then $G$ admits a saturated matching if and only if
\[
\forall\, S\subseteq X,\qquad |N(S)| \ge |S|.
\]

\end{theorem}

\begin{remark}
  \label{R:regularBipartiteSatisfiesHall}
Note that a regular bipartite graph satisfies the hypothesis of Hall's matching theorem and thus has a perfect matching.
Indeed, suppose that $G=(V,E)$ is a bipartite and $k$-regular with $V = X \amalg Y$. 
Let $S \subset X$. 
Then since $G$ is $k$-regular we know that there are $k \vert S \vert$ incident to $N(S)$.
If $\vert N(S) \vert < \vert S \vert$, then by the pigeonhole principle there would exist at least one edge of degree greater than $k$ which gives a contradiction.  
\end{remark}

\begin{lemma}\label{lem:exists-loop}
	Let $G$ be a directed graph on $n$ vertices with no loops. Suppose that every vertex has out-degree one. 
	Then $G$ has a directed cycle. 
\end{lemma}
\begin{proof}
	We will create a walk on the graph. 
	Start with a vertex and then follow the unique edge pointing out. 
	It will eventually revisit a vertex it once visited or make a cycle of exactly length $n$.
	A path starting from the previously vertex defines a cycle. 
\end{proof}

In the proof where we fill the gap we make use of a ``third form'' (which is really equivalent to the second form). We introduce this now to make the presentation easier to follow. 
\begin{remark}[Ritt's ``third'' form]
	Suppose that $A$ is a matrix in Ritt's second form (in particular, $\tdet(A)= a_{n,1} + a_{1,n} + \sum_{i=2}^{n-1} a_{i,i}$). 
	Then by cyclically permuting columns $2,3,\ldots,n$ we obtain a matrix $B$ with 
	\begin{equation}\label{E:third-form-1}
		\tdet(B) = b_{n,1} + \sum_{i=2}^n b_{i,i+1}.
	\end{equation}
	Explicitly $b_{i,j} = a_{i,\rho(j)}$ where $\rho$ is the inverse of the permutation $2\mapsto 3 \mapsto \cdots \mapsto n \mapsto 2$. 
	Note that in this situation for $A$ to be in Ritt's second form we require that $\tdet(\widetilde{A}_{n,n})= \sum_{i=1}^{n-1} a_{i,i}$. 
	Under the transformation above, since the $n$th column of $A$ corresponds to the 2nd column of $B$, this is equivalent to \begin{equation}\label{E:third-form-2}
		\tdet(B_{n,2}) = b_{1,1} + \sum_{i=2}^{n-1} b_{i,i+1}.
	\end{equation}
	Hence a matrix $B$ which satisfies \eqref{E:third-form-1} and \eqref{E:third-form-2} could be said to be in \defi{Ritt's third form}, which is equivalent to a matrix being in the second form. 
	This will be useful in what follows. 
\end{remark}

One last thing before filling the gap.  
We want to record a formula (equation \eqref{E:cofactor}) for taking transversals of cofactor matrices of $A$. 
In the matrix $\widetilde{A}_{r,c}$ we are omitting the $r$th row and $c$th column. 
This means that the tropical determinant will now take the maximum over functions $\rho \colon \lbrace 1,2,\ldots,n \rbrace \setminus \lbrace r \rbrace \to \lbrace 1,2,\ldots, n \rbrace \setminus \lbrace c \rbrace$ and for such a $\rho$ a transversal now takes the form 
\begin{equation}\label{E:cofactor}
	\tdet(\widetilde{A}_{r,c})= a_{1,\rho(1)} + \cdots + a_{r-1,\rho(r-1)} + a_{r+1,\rho(r+1) } + \cdots + a_{n,\rho(n) }.
\end{equation}
We note that $c$ in the $j$-position and $r$ in the $i$-position must be omitted. 
It is helpful to note that every such function can be extended to a permutation $\widetilde{\rho}\in S_n$ satisfying $\widetilde{\rho}(r)=c$. 

We can now prove that we can put any matrix into the second form filling the gap in Ritt's proof of the linear case. 
\begin{proof}[Proof of Proposition~\ref{P:second-form}]

  Let $A = (a_{i,j})$ be the order matrix.
 Swapping rows and columns, we can always assume that $a_{n,1}$ is the maximum of its column and that $a_{2,2} + \cdots + a_{n-1,n-1} + a_{1,n}$ is a maximal transversal of $\widetilde{A}_{n,1}$.
Apply the inverse of the permutation $2 \mapsto 3 \mapsto \cdots \mapsto n-1 \mapsto n$ to obtain a matrix $B$ such that $b_{1,2} + b_{2,3} + \cdots + b_{n-1,n}$ is a maximal transversal of $\widetilde{B}_{n,1}$. 
Replacing $A$ with $B$ we can assume without loss of generality that $a_{1,2}+a_{2,3} + \cdots + a_{n-1,n}$ is a maximal transversal of $\widetilde{A}_{n,1}$. 

Let $D$ be the matrix given by swapping rows $1$ and $i$ and columns $2$ and $i+1$ of $A$. We claim that $D$ is in Ritt's third form, i.e., the transversal
\begin{equation}
\label{eq:swapped-transversal}
  d_{1,1} + d_{2,3} + \cdots + d_{i-1,i} + d_{i,i+1} + d_{i+1,i+2} +\cdots + d_{n-1,n}
\end{equation}
of $\widetilde{D}_{n,2}$ is maximal, if and only if
\begin{equation}
\label{eq:shifted-transversal}  
(a_{1,2} + a_{2,3} + \cdots + a_{i-1,i}) +a_{i,1}+ (a_{i+1,i+2} + \cdots + a_{n-1,n})  
\end{equation}
  is a maximal transversal of the submatrix $\widetilde{A}_{n,i+1}$. To prove this claim, let $C$ be the matrix given by swapping rows $1$ and $i$ of $A$, i.e.,
$$c_{rs}= \begin{cases}
	a_{is} & \text{ if } r = 1 \\
	a_{1s} & \text{ if } r = i\\
  a_{rs} & \text{ otherwise. }
  \end{cases}
  $$  
Then $D$ is the matrix given by swapping columns $2$ and $i+1$ of $C$, i.e.,
$$d_{rs}= \begin{cases}
	c_{r(i+1)} & \text{ if } s = 2 \\
	c_{r2} & \text{ if } s = i+1\\
  c_{rs} & \text{ otherwise. }
  \end{cases}
  $$  
  Combining these two formulas gives that 
\[
  d_{1,1} + d_{2,3} + \cdots + d_{i-1,i} + \boxed{d_{i,i+1}} + d_{i+1,i+2} +\cdots + d_{n-1,n}
\]
is equal to 
\[
  \underline{c_{1,1}} + c_{2,3} + \cdots + c_{i-1,i} + \boxed{c_{i,2}} +c_{i+1,i+2} + \cdots + c_{n-1,n},
  \]
  which is equal to
\[
  \underline{a_{i,1}} + a_{2,3} + \cdots + a_{i-1,i} + \boxed{a_{1,2}} +a_{i+1,i+2} + \cdots + a_{n-1,n}
  \]  
(where for clarity we have underlined and boxed the terms that change). Rearranging the order of this sum by switching $a_{i,1}$ and $a_{1,2}$ gives
\[
\left(\boxed{a_{1,2}} + a_{2,3} + \cdots + a_{i-1,i}\right) + \underline{a_{i,1}} + \left(a_{i+1,i+2} + \cdots + a_{n-1,n}\right).
\]
In particular, since our operations take Equation \eqref{eq:shifted-transversal} to Equation \eqref{eq:swapped-transversal}, Equation \eqref{eq:swapped-transversal} is a maximal transversal of $\widetilde{D}_{2,n}$ (i.e., $D$ is in Ritt's third form) if and only if Equation \eqref{eq:shifted-transversal}  is a maximal transversal of $\widetilde{A}_{n,i+1}$, as claimed.

It now suffices to prove that for some $i$, Equation \eqref{eq:shifted-transversal}  is a maximal transversal of $\widetilde{A}_{n,i+1}$. 
We proceed by contradiction. Suppose that for all $i \in \{1,\ldots,n-1\}$, Equation \eqref{eq:shifted-transversal}  is not a maximal transversal of $\widetilde{A}_{n,i+1}$.
This means that for each $i$ there exists a larger transversal, i.e.,  a bijection $\rho_i\colon \lbrace 1,2,\ldots, n-1\rbrace \to \lbrace 1,2,\ldots, n \rbrace \setminus \lbrace i \rbrace$ such that 
\begin{equation}
  \label{eq:second-form-contradiction-inequality}  
	 (a_{1,2} + a_{2,3} + \cdots + a_{i-1,i}) +a_{i,1}+ (a_{i+1,i+2} + \cdots + a_{n-1,n}) < \sum_{r=1}^{n-1} a_{r,\rho_i(r)}.
\end{equation}

  Consider the directed graph whose vertices are the integers $1,\ldots,n-1$, with an edge from $i$ to $j$ if $a_{i,1}$ appears on the left hand side of Equation \eqref{eq:second-form-contradiction-inequality} and $a_{j,1}$ appears on the right hand side of Equation \eqref{eq:second-form-contradiction-inequality}.
Since $\rho_i$ is a bijection and $1$ is in the image, there exists a unique $r$ such that $\rho_i(r)=1$; in particular, each vertex has a unique outgoing edge.
By Lemma \ref{lem:exists-loop}, this graph has a cycle $C = i_1i_2\ldots i_ki_1$. Let $\sigma \in S_n$ be the cyclic permutation given by $\sigma(i) = i$ if $i$ is not in the support of the cycle, $\sigma(i_{k}) = i_1$, and $\sigma(i_j) = i_{j+1}$ for $i < k$.

From here, we will only consider Equation \eqref{eq:second-form-contradiction-inequality} for $i$ in the support of $C$.
To ease notation, we apply the permutation $(1i_1)(2i_2)\cdots(ki_k)$ to the rows and columns of $A$, and without loss of generality can assume that $\sigma = (12\ldots k)$ (and still that $a_{n,1}$ is the maximum of its column and $a_{1,2} + a_{2,3} + \cdots + a_{n-1,n}$ is a maximal transversal of $\widetilde{A}_{n,1}$).

Now consider the sum of the inequalities from Equation \eqref{eq:second-form-contradiction-inequality} with $i = 1,\ldots,k$, and add $a_{1,2} + \cdots + a_{k,k+1}$ to each side. Then the left hand side of this sum is
$$
(a_{1,1} + \cdots + a_{k,1}) + k(a_{1,2} + a_{2,3} + \cdots + a_{n-1,n});
$$
we note that the second term is $k$ times the given maximal transversal of $\widetilde{A}_{n,1}$. 

The right hand side is more complicated.
	It is the sum 
	\begin{equation}\label{E:graph-sum-0}
		\left(\sum_{r=1}^k a_{r,r+1}\right) + \sum_{i=1}^k \sum_{r=1}^{n-1} a_{r,\rho_i(r)}.
	\end{equation} 
We will describe this sum using a bipartite graph, reduce this graph to a $k$-regular graph by cancelling terms on either side of the inequality which will result in a $k$-regular bipartite graph. 
We then use Hall's Matching Theorem to show that the left over pieces are sums of transversals which we can then bound using Ritt's cycle trick, giving a contradiction.

We start by constructing a bipartite graph $G_0$. We form $G_0$ with vertices $X=\lbrace x_1,\ldots,x_{n-1}\rbrace$ and $Y=\lbrace y_1,\ldots,y_n\rbrace$ where we add an edge $(x_i,y_j)$ to $G_0$ provided $a_{i,j}$ in the left hand side. 
In this notation we have 
\begin{equation}\label{E:graph-sum-0}
	\left(\sum_{r=1}^k a_{r,r+1}\right) + \sum_{i=1}^k \sum_{r=1}^{n-1} a_{r,\rho_i(r)} = \sum_{i=1}^{n-1}\sum_{j=1}^n c_{i,j}^0 a_{i,j}
\end{equation} 
where $c_{i,j}^0$ is the number of edges from $x_i$ to $y_j$ in $G_0$.

We can now compute the out-degree of each $x\in X$ and the in-degree of each $y \in Y$.
Since $\rho_i\colon \lbrace 1,\ldots,n-1\rbrace \to \lbrace 1,\ldots,n\rbrace \setminus \lbrace i \rbrace$ are bijections, we know that $\rho_i$ contributes $1$ to the out-degree of $x$ for each $x \in X$ and contributes one to the in-degree $y$ for each $y\in Y$ except for $y_j$. 
With only these contributions the out-degree of $x$ would be $k$ for each $x\in X$ and the in-degree of each $y\in Y$ would be $k-1$.
But we also have the additional contribution of $\sum_{r=1}^k a_{r,r+1}$ which contributes one to the out-degree of each $x_1,\ldots,x_k$ and one to the in-degree of $y_2,\ldots,y_{k+1}$. 

Hence 
$$ \deg_{\out}(x_i) = \begin{cases}
	k+1, & 1\leq i \leq k \\
	k, & k<i\leq n-1
\end{cases}, \qquad \deg_{\IN}(y_j) = k, \quad 1 \leq j \leq n.$$

Due to our choice of $\sigma$, the terms $a_{1,1},\ldots, a_{k,1}$ each appear exactly once on the right hand side of the sum (i.e. $c_{1,1}^0=c_{2,1}^0= \cdots = c_{k,1}^0=1$). 
They also appear exactly once on the left hand side of our inequality. 
Thus $a_{1,1} + \cdots + a_{k,1}$ cancels from both sides of the resulting inequality. 

In terms of graphs, this amounts to deleting the vertex $y_1$ to obtain a new graph $G$ where the $\deg_{\out}(x_i)$ is lowered by one for $1\leq i \leq k$ (while all of the other vertex degrees remain the same).
This makes $G$ a $k$-regular bipartite graph. 

We will now let $c_{i,j}$ be the number of edges from $i$ to $j$ in $G$ so that we now have 
 $$ k(a_{1,2} + a_{2,3} + \cdots + a_{n-1,n}) \leq \sum_{(i,j)\in E(G)} c_{i,j} a_{i,j}.$$

By Hall's Matching theorem $G$ admits a perfect matching $M_1$. Removing the edges in $M_1$ from $G$ gives another regular bipartite graph; this graph again has a perfect matching $M_2$. Inductively, there exist perfect matchings $M_1$, \ldots, $M_k$ such that $E(G) = \cup^k_i M_l$. 

For each matching $M_l$, we defined bijections $\tau_l\colon \{1,\ldots,n-1\} \to \{2,\ldots,n\}$ defined by $\tau_l(i)=j$ if $x_iy_j \in M_l$. Then we have 
	$$
\sum_{(i,j)\in E(G)} c_{i,j} a_{i,j}  = \sum^k_{l=1} \sum_{i=1}^{n-1} a_{i,\tau_l(i)}.
$$

Each summand $\sum_{i=1}^{n-1} a_{i,\tau_l(i)}$ is a transversal of $\widetilde{A}_{n,1}$, so by Ritt's cycle trick (Lemma \ref{lem:ritt-permutation}),
\[
  \sum_{i=1}^{n-1} a_{i,\tau_l(i)} \leq a_{1,2} + a_{2,3}+\cdots + a_{n-1,n}.
\]
Putting everything together we get
\begin{align*}
	 k(a_{1,2} + a_{2,3}+\cdots + a_{n-1,n}) &< \sum_{(i,j)\in E(G)} c_{i,j} a_{i,j} \\
	 &=\sum^k_{l=1} \sum_{i=1}^{n-1} a_{i,\tau_l(i)} \\
	 &\leq \sum^k_{l=1} \left(a_{1,2} + a_{2,3}\cdots + a_{n-1,n}\right) \\
	 &= k \left(a_{1,2} + a_{2,3}\cdots + a_{n-1,n}\right).
\end{align*}
which is a contradiction.
\end{proof}

\section{Ritt's Proof of The Linear Case}\label{S:linear-case}

We proceed by a double induction on the number of variables and the order matrix.
Let $\bar{\RR} = \RR\cup \lbrace -\infty \rbrace$. 
\begin{definition}
\label{D:ritt-ordering}
	We define \defi{Ritt's matrix ordering} in the following way.
	\begin{enumerate}
		\item For $a=(a_1,\ldots,a_n)\in \bar{\RR}^n$ we write $\widetilde{a} \in \bar{\RR}^n$ for the permutation of the entries of $a$ so that the entries are in non-decreasing order. We say $a \prec b$ if and only if $\widetilde{a}<_{\lex} \widetilde{b}$.\footnote{In \cite{Ritt1935}, Ritt called these  ``characteristic sets'', which of course disagrees with the modern definition in \ref{D:characteristic-set}.}
		\item Let $A,B\in M_n(\bar{\RR})$.
		Write matrices as a vector of column vectors $A=[a_1|\cdots|a_n], B=[b_1| \cdots | b_n]$ so that $A,B \in (\bar{\RR}^n)^n$. 
		We write $A\prec B$ if and only if $A \prec_{\lex} B$.
	\end{enumerate}
\end{definition}
Note that if you view a matrix as a vector of column vectors, then the matrix ordering Ritt defines is the same as lexicographic ordering on the columns (where the columns are understood to be ordered by $\prec$).

\begin{lemma}\label{L:minimal-prime-after-division}
	\label{lemma:nonvanishing-separants}
	Let $u_1,u_2, \ldots,u_n  \in K\lbrace x_1,\ldots,x_n\rbrace$.
	Let $P\supset [u_1,u_2, \ldots,u_n]$ be a minimal prime, corresponding to a component $\Sigma_1$.
	Let $h = s u_2 - Q u_1$ be the Ritt remainder after $\partial$-division of $u_2$ by $u_1$ in the variable $x_1$.
	Suppose that $s \notin P$ (i.e., $S$ does not vanish on $\Sigma_1$).

	Then $P \supset [u_1,h,u_3,\ldots,u_n]$ is a minimal prime.
  \end{lemma}

In other words, if we perform a Ritt division and the separant $s$ does not vanish on $\Sigma_1$, then $\Sigma_1$ is still an irreducible component of the new system.  

\begin{proof}
  
	First observe that $h \in [u_1,u_2,\ldots,u_n] \subset P$ so $[u_1,h,u_3,\ldots,u_n] \subset [u_1,u_2,\ldots,u_n] \subset P$. 
	Suppose that $P$ is not a minimal prime over $[u_1,h,u_3,\ldots,u_n]$ and that there is some prime $P'$ such that $[u_1,h,u_3,\ldots,u_n] \subset P' \subsetneq P$. 
	We have that $h + Qu_1 = s u_2 \in [u_1,h,u_3,\ldots,u_n]$ and since $s \in R\setminus P=S_P$ this implies that $u_2 \in \sat_{S_P}([u_1,h,u_3,\ldots,u_n])$. In particular,  $\sat_{S_P}([u_1,h,u_3,\ldots,u_n])=\sat_{S_P}([u_1,u_2,\ldots,u_n])$; since $\sat_{S_P}(P)$ (resp.~$\sat_{S_P}(P')$) is still minimal over $\sat_{S_P}([u_1,u_2,\ldots,u_n])$ (resp.~$\sat_{S_P}([u_1,h,\ldots,u_n])$), we must have that $\sat_{S_P}(P') = \sat_{S_P}(P)$. Since $s \not \in P$, we conclude that $P' = P$. 
\end{proof}

With the added ingredient we can now prove the linear case of the Jacobi Bound Conjecture. 
This is a Corollary of our reduction process. 
\begin{corollary}\label{C:linear-case}
	The Jacobi Bound Conjecture holds in the linear case. 
\end{corollary}
\begin{proof}
The proof is a nested induction. 
The outer induction is on $n$, the number of equations which is equal to the number of variables.
The inner induction is on the order matrices according to Ritt's ordering (Definition~\ref{D:ritt-ordering}).

The idea is that we can always do reductions to lower the system in a way that preserves the irreducible component and we can do this until we get to an order matrix that takes the form \eqref{E:reduced}. 
When we get to this stage we get to induct on the number of variables. 
\\

\noindent \emph{Base Case (Outer Induction):}
  If $n = 1$, JBC is known by the Ritt Bound \cite{Ritt1935}.
 \\
  
\noindent \emph{Inductive Step (Outer Induction):}
We prove the inductive step by induction. 
\\\

\noindent \emph{Base Case (Inner Induction):}
Suppose that some column of $A$ has exactly one non $-\infty$ entry.
Then we can apply the outer inductive hypothesis, as follows.
Swapping rows and columns, we may assume that $A$ is in the form
\begin{equation}\label{E:reduced}
	A = \begin{pmatrix}
		a_{1,1} & a_{1,2} & \cdots &a_{1,n} \\
		-\infty & a_{2,2} & \cdots & a_{2,n} \\
		\vdots & \vdots & \ddots  & \vdots \\
		-\infty & a_{n,2} & \cdots & a_{n,n}
	\end{pmatrix} 
\end{equation} 
and that $a_{1,1} \neq -\infty$.

Suppose we have a reduced system with order matrix as in \eqref{E:reduced}. 
Let $P_1\supset [u_1,\ldots,u_n]$ be a minimal prime with $\trdeg_K(\kappa(P_1))<\infty$. 
We will show how to conclude that $\trdeg_K(\kappa(P_1))\leq J(u_1,u_2,\ldots,u_n)$. 
Let $F= \kappa(P_1)$ and let $\bar{x}_1$ be the image of $x_1$ in $\kappa(F)$. 
Consider now $F_1 = K(\lbrace \bar{x}_1 \rbrace) \subset F$ the $\partial$-field $\partial$-generated by $\bar{x}_1$ over $K$. 
Let $w_2,\ldots, w_n \in F_1\lbrace x_2,\ldots,x_n\rbrace$ be given by  $w_i = u_i(\overline{x}_1,x_2,\ldots,x_n)$.
Let $P_1'$ be the kernel of the map $\xi\colon F_1\lbrace x_2,\ldots,x_n\rbrace \to \kappa(P_1)$.

We have that $\kappa(P_1)=\kappa(P_1')$ so $\trdeg_K(P_1')<\infty$.
We claim that $P_1' \supset [w_2,\ldots,w_n]$ is minimal. 
Since $u_2,\ldots,u_n$ do not involve $x_1$ or its derivatives, the order matrix of the system $w_2,\ldots,w_n$ is
$$A_1=\begin{pmatrix}
	a_{2,2} & \cdots & a_{2,n} \\
	\vdots & \ddots  & \vdots \\
	a_{n,2} & \cdots & a_{n,n}
\end{pmatrix}.$$
By induction on $n$ and we know that JBC holds for this system provided we can show that $P_1 \supset  [w_2,\ldots,w_n]$ is minimal. 

Suppose $P_1' \supset [w_2,\ldots,w_n]$ is not minimal and that there is some $Q_1'$ such that $P_1' \supsetneq Q_1' \supset [w_2,\ldots,w_n]$. 
Consider the maps 
$$\varphi\colon K\lbrace x_1,\ldots,x_n\rbrace \to F_1\lbrace x_2,\ldots,x_n\rbrace/Q_1', \quad \psi\colon F_1\lbrace x_2,\ldots,x_n\rbrace \to F_1\lbrace x_2,\ldots,x_n\rbrace/Q_1'.$$ 
$$ \sigma\colon K\lbrace x_1,\ldots,x_n\rbrace \to F_1\lbrace x_2,\ldots,x_n\rbrace, \quad \psi \sigma = \varphi.$$
Let $\ker(\varphi)=Q_1$. 
We have that $Q_1 \supset [u_1,u_2,\ldots,u_n]$.
By minimality of $P_1$ we have that $P_1 = Q_1$. 
This means that $\varphi$ is really just reduction modulo $P_1$ and hence factors as $\varphi=\xi\sigma$. 

Let $f \in P_1'\setminus Q_1'$. 
Write it as $f=\sum_{\alpha} (\bar{a}_{\alpha}/\bar{b}_{\alpha}) y^{\alpha}$ where $y=(x_2,\ldots,x_n)$, and $\alpha \in \supp(f) \subset \ZZ_{\geq 0}[\partial]^{n-1}$ where $a_{\alpha},b_{\alpha} \in K\lbrace x_1\rbrace$ and $\bar{a}_{\alpha}$ and $\bar{b}_{\alpha}$ are evaluated at $\bar{x}_1$. 
Let $b=\prod_{\alpha} b_{\alpha} \in K\lbrace x \rbrace$.
Let $g = \sum_{\alpha} c_{\alpha} y^{\alpha} \in K\lbrace x_1,\ldots,x_n\rbrace$ where $c_{\alpha}=b (a_{\alpha}/b_{\alpha}) \in K\lbrace x_1 \rbrace$.
On one hand $0\neq \sigma(b) =\varphi(b) \in F_1$ but $\varphi(gb) = \psi(\sigma(g)\sigma(b))=\psi(f/\sigma(b))=\psi(f)/\sigma(b)\neq0$ and $\psi(f)\neq 0$. 
On the other hand $\varphi(gb)=\xi( \sigma(g)\sigma(b))=\xi( f/\sigma(b))=\xi(f)/\sigma(b)=0$ ($\xi(f)=0$ since $f \in P_1'$) which is a contradiction and hence $P_1'$ is a minimal prime ideal.

Since we have shown that $P_1'$ is minimal and finite transcendence degree over $F_1$ by induction we have $\trdeg_{F_1}(\kappa(P_1'))\leq J(u_2,\ldots,u_n)$.
But $\trdeg_K(\kappa(P_1)) = \trdeg_{F_1}(\kappa(P_1))+\trdeg_K(F_1) \leq J(u_2,\ldots,u_n) + a_{1,1} = J(u_1,u_2,\ldots,u_n)$ which proves the result. 
The first equality here uses additivity of transcendence degrees in towers of fields. 
\\

\noindent \emph{Inductive Step (Inner Induction)}:
The induction here is on the order matrices with the Ritt ordering as in Definition~\ref{D:ritt-ordering} with the base case being matrices of the form \eqref{E:reduced}.
Suppose that the system has an order matrix $A$ and that the Jacobi Bound conjecture holds for any system with an order matrix $B \prec A$ in Ritt's matrix ordering (Definition~\ref{D:ritt-ordering}).

If a column has at least two non $-\infty$ entries, then by Propositions~\ref{P:first-form} and \ref{P:second-form} we can put the system into Ritt's first or second form.
Then Ritt division as in Propositions~\ref{P:ritt-first} and \ref{P:ritt-second} give a new system $v_1,\ldots,v_n$ such that $J(v_1,\ldots,v_n) \leq J(u_1,\ldots,u_n)$ and whose order matrix is lower with respect to the Ritt ordering on matrices.
Since the system $u_1,\ldots,u_n$ is linear, all of its separants are units.
In particular, by  Lemma \ref{L:minimal-prime-after-division}, $\Sigma_1$ is still an irreducible component of the new system i.e. $P_1 \supset [v_1,\ldots,v_n]$ is a minimal prime ideal.
By induction, JBC is true for the new system, so $\dim \Sigma_1 \leq J(v_1,\ldots,v_n)$.
Since $J(v_1,\ldots,v_n) \leq  J(u_1,\ldots,u_n)$, we conclude that $\dim \Sigma_1 \leq J(u_1,\ldots,u_n)$, i.e., that JBC is true for $\Sigma_1$.
\end{proof}
	
	\section{Generalities on Differential Linear Series}\label{S:linear-series}

  In the nonlinear case, when performing a Ritt division the separant hypothesis of Lemma \ref{L:minimal-prime-after-division} generally fails.
  In particular, while $\Sigma_1$ will still be a subvariety of the new system, it might no longer be an irreducible component, creating a serious obstruction to performing a clean induction.
  In Ritt's proof of the two variable case of JBC in, he introduces a degeneration that we call the ``Ritt pencil trick''.
  This pencil has a fiber that is ``good enough'' to continue inducting, as we explain below.
  Ritt's proof does not cleanly generalize to more than two variables.
  Below we prove that if one assume's the Dimension Conjecture, then an argument similar to Ritt's proof works for the general JBC.

	In analogy with classical procedures in algebraic geometry (see e.g.~\cite[III.2.8]{Eisenbud1995}, \cite[Lecture 4]{Harris1995}) we consider families of differential algebraic varieties.
	
	\begin{definition}
		Let $\AA^n_{\infty}=\Spec K\lbrace x_1,\ldots,x_n\rbrace$ and $\AA^m_{\infty} = \Spec K\lbrace y_1,\ldots,y_m\rbrace$ and write $x=(x_1,\ldots,x_n)$ and $y=(y_1,\ldots,y_m)$,
		By a \defi{differential linear series} we will mean a family of $D$-schemes $f\colon\Gamma \to \AA^m_{\infty}$ parametrized by $\AA^m_{\infty}$ that are linear combinations of $\partial$-polynomials.
		$$\begin{cases}
			v_1(x,y) = v_{1,1}(x)y_1 + \cdots + v_{1,m}(x) y_m =0,\\
			v_2(x,y) = v_{2,1}(x)y_1 + \cdots + v_{2,m}(x) y_m=0,\\
			\qquad \vdots  \\
			v_e(x,y) = v_{e,1}(x)y_1 + \cdots + v_{e,m}(x) y_m=0.
		\end{cases}$$
		where $v_{i,j} \in K\lbrace x \rbrace$ for $1\leq i \leq e$ and $1\leq j \leq m$. In other words 
		$$ \Gamma = \Spec K\lbrace x,y \rbrace /[v_1(x,y),v_2(x,y),\ldots,v_n(x,y)].$$ 
		
		By a \defi{reduced differential linear series} we mean a $D$-scheme $\Gamma_0$ of the form $\Gamma_0 = \Gamma_{\red}$ for $\Gamma$ a linear series. 
	\end{definition}
	To keep terminology simple we refer to ``differential linear series'' and ``reduced differential linear series'' simply as ``linear series'' and ``reduced linear series''.

Example~\ref{E:bad-fiber} shows that these families are not flat in general. 
 In particular deforming away a single coefficient creates a differential algebraic variety which is generically of absolute dimension 3 to a differential scheme of infinite absolute dimension.
\begin{example}\label{E:bad-fiber}
	Consider $\Gamma$ given by $\Spec K\lbrace x, y,z \rbrace/[u,v]$ which we view as a family over $\Spec K\lbrace z \rbrace$ with $u=x''+y', v=zy + x''+y'$. 
	We have $\dim(\Gamma_0)=\infty$ and if $\eta$ is the generic point of $\Spec K\lbrace z \rbrace$, then $\dim(\Gamma_{\eta})=3$. 
\end{example} 
	
	\begin{remark}
		Suppose that $K$ is differentially closed. 
		In the case $e=1$, we have 
		$$ v_1(x,y) = v_1(x)y_1+\cdots + v_m(x) y_m=0.$$
		There is a surjective map
$$\operatorname{Span}_K(\lbrace v_1,\ldots,v_n\rbrace ) \to 
\lbrace \Gamma_{b} \subset K^n \colon b \in K^n \rbrace,$$
where $\operatorname{Span}_K(\lbrace v_1,\ldots,v_n\rbrace ) \subset K\lbrace x_1,\ldots,x_n\rbrace$ and where $\operatorname{Span}$ denotes the $K$-vector space span just as in the classical setting.
	\end{remark}

	In analogy with \cite[Definition 1.1.2]{Lazarsfeld2004} we make the following definition. 
	\begin{definition}
		Let $f\colon \Gamma \to \AA^m_{\infty} = \Spec K\lbrace y_1,\ldots,y_m\rbrace$ be a linear series given  by $v_i = \sum_{j=1}^m v_{i,j}(x) y_j=0$ for $1\leq i \leq e$. 
		\begin{enumerate}
			\item The \defi{base ideal} $\frak{b}$ of the linear series is 
			$$ \mathfrak{b} = [v_{i,j} \colon 1 \leq i \leq e, 1 \leq j \leq m ] \subset K\lbrace x_1,\ldots,x_n\rbrace.$$
			\item The \defi{base locus} $\BS(\Gamma)=V(\frak{b}) \subset \AA^m_{\infty}$.
		\end{enumerate}
	\end{definition}
	Geometrically the base locus is the closed $D$-subscheme of $\AA^n_{\infty}$ that is contained in $\Gamma_{b}$ for all $b$. 
	We will adopt the standard abuse of notation and refer to both the trivial family $\BS(\Gamma)\times \AA^m_{\infty} \subset \Gamma$ and $\BS(\Gamma)\subset \AA^n_{\infty}$ as ``the base locus of $\Gamma$''.
	This will allow us to talk both about the ``base locus as subscheme of $\AA^n_{\infty}$'' and ``a component of $\Gamma$ contained in the base locus''.

	\begin{definition}
		Let $\Gamma \to \AA^m_{\infty}$ be a linear series.
		Let $\Gamma_1 \in \Irr(\Gamma)$.  
		\begin{enumerate}
			\item We say $\Gamma_1$ is \defi{fixed component} if $\Gamma_1 \subset \BS(\Gamma)\times \AA^m_{\infty}$.
			\item If $\Gamma_1$ is not fixed, then it is called a \defi{moving component}.
		\end{enumerate} 
	\end{definition}
	Let $\mu \in K$. 
	Every component of $\Gamma_{\mu}$ is contained in a component of $\Gamma$.

	\begin{definition}
		Let $\Gamma$ be a reduced linear series. 
		We will write $\Gamma = \Gamma_{\fix} \cup \Gamma_{\move}$ where  
		\begin{enumerate}
			\item $\Gamma_{\fix}$ is \defi{fixed part} of $\Gamma$ and is defined to be the union of irreducible components which are fixed, and 
			\item $\Gamma_{\move}$ is the \defi{moving part} and is defined to be the union of irreducible components of $\Gamma$ which are moving.
		\end{enumerate}		
	\end{definition}
	In terms of ideals, $I_{\Gamma_{\fix}}$ is the intersection of those minimal primes above $I_{\Gamma}$ which contain $\frak{b}$.

	\begin{definition}
		A component $\Gamma_{\mu,1} \in \Irr(\Gamma_{\mu})$ is called \defi{fixed} if there exists some fixed component $\Gamma_1\in \Irr(\Gamma)$ such that $\Gamma_{\mu,1} \subset \Gamma_1$. 
		If it is not fixed it is called \defi{moving}. 
	\end{definition}

The following theorem is one of the two theorems that ``gets us out of'' degenerate situations in our inductive proof of the Jacobi Bound Conjecture.
Our particular application will be ``extra degenerate''; not only will our separants be generically vanishing on our component $\Sigma_1$, but that component will actually be a component of the base loci.
Fortunately, this ``extra degenerate'' behavior is surprisingly nice with regard to the Jacobi Bound Conjecture --- it puts us in a situation where our Jacobi Bound has decreased and where the irreducible differential scheme we are trying to apply the Jacobi Bound Conjecture to hasn't changed at all!
	\begin{theorem}\label{T:fixed}
	Let $\Sigma \subset \AA^n_{\infty}$ be a differential scheme of $\partial$-finite type and let $\Gamma \to \AA^1_{\infty}$ be a pencil of differential subschemes of $\AA^n_{\infty}$.  
  Let $\Sigma_1=V(P)$ be an irreducible component of $\Sigma$. 	
  Suppose $\Sigma_1 \subset \BS(\Gamma) \subset \Sigma$.
  	\begin{enumerate}
  		\item Then $\Sigma_1$ is a component of $\BS(\Gamma)$.
			\item Fix $\mu \in K$. If $\Gamma_{\mu,1}$ is the component of $\Gamma_{\mu}$ containing $\Sigma_1$ and $\Gamma_{\mu,1}$ is fixed, then  $\Gamma_{\mu,1} =\Sigma_1$. 
		\end{enumerate}
	\end{theorem}
	\begin{proof}
		Let $\Sigma$ be defined by the ideal $[u_1,\ldots,u_e] \subset K\lbrace x_1,\ldots,x_n\rbrace$. 
		Let $P_1$ be a minimal prime over $[u_1,\ldots,u_e]$ defining $\Sigma_1$.
		Let $\Gamma$ be defined by the ideal $[v_1,\ldots,v_e] \subset K\lbrace x_1,\ldots,x_n,y\rbrace$.
		\begin{enumerate}
			\item []
			\item 	Let $\frak{b}$ be the base ideal of the pencil. We have that $[u_1,\ldots,u_e] \subset \frak{b}$ by hypothesis so that our original differential algebraic variety contains the base locus. 
			We also know that the base locus contains a component of our original system: we have that $[u_1,\ldots,u_n] \subset \frak{b} \subset P$. 
			If $P$ were not minimal over $\frak{b}$ there would exist some prime $\widetilde{P}$ with $P \supsetneq \widetilde{P} \supset \frak{b}.$ 
			This would imply the following containments  
			$$[u_1,\ldots,u_e] \subset \frak{b} \subset \widetilde{P} \subsetneq P.$$
			But this would imply that $P$ was not minimal over $[u_1,\ldots,u_e]$. 
			This is a contradiction and hence $\Sigma_1$ must be a component $\BS(\Gamma)$. 
			\item 
			Let $\widetilde{\Gamma}_1=\Gamma_{\mu,1}$ be a component containing $\Sigma_1$.
			Let $\widetilde{P}$ be such that $\widetilde{\Gamma}_1=V(\widetilde{P})$. 
			Let $\Sigma_1 =V_{\partial}(P)$. 
			This implies that $\widetilde{P}\subset P$ and that $\widetilde{P}$ is minimal over $\frak{b}$. 
			We know that $[v_1(x,\mu),v_2(x,\mu),\ldots,v_e(x,\mu)] \subset \frak{b} \subset \widetilde{P} \subset P$ (the first inclusion follows from the fact that every fiber contains the base locus --- in fact, that is the definition of the base locus). Since $P$ is minimal over $\frak{b} \cap K\lbrace x_1,\ldots,x_n\rbrace$ by the previous part we must have that $\widetilde{P}=P$. 
		\end{enumerate}
	\end{proof}
	
\section{Ritt's Pencil Trick}
\label{S:ritt-pencil}

\begin{definition}\label{D:degenerate}
	Let $\Sigma = V([u_1,\ldots,u_n])$ and let $\Sigma_1$ be an irreducible component with generic point $\eta_1$. 
	Suppose that we have ordered $u_1,\ldots,u_n$ and $x_1,\ldots,x_n$ so that the order matrix is in Ritt's first form or second form. 
	Let $\ell$ be $x_1^{(r)}$, where $r =\ord^{\partial}_{x_1}(u_1)$. 
	We say the \defi{situation is degenerate} at $\Sigma_1$ if the separant vanishes generically on $\Sigma_1$. 
	In equations:
	\begin{equation}
  \label{eq:ritt-pencil-s-t}
		s_1=\frac{\partial u_1}{\partial \ell}(\eta_1)=0.
	\end{equation} 
\end{definition}
Note that if $\Sigma_1$ is associated to the prime ideal $P_1$ the separant generically vanishing is the same as the containment $s_1 \in P_1$. 

\begin{definition}\label{D:ritt-pencil}
	Let $\Sigma=V([u_1,\ldots,u_n])$. Let $\Sigma_1 \in \Irr(\Sigma)$ with generic point $\eta_1$. 
	Suppose we are in a degenerate situation for $u_1$ at $\eta_1$.  
  The \defi{Ritt pencil} is the family $\Gamma  \to \AA^1_{\infty} = \Spec K\lbrace y \rbrace $ of $D$-schemes defined by 
	$$\Gamma = \Spec(K\lbrace x_1,\ldots,x_n, y \rbrace/ [t_1 + y s_1, u_2,\ldots,u_n ]),$$
	where $s_1 = \partial u_1/\partial \ell$ where $\ell$ is the leader of $u_1$ in the variable $x_1$ and $t_1$ is defined by the equation 
	\begin{equation}\label{E:t1-definition}
		d\cdot u_1 = t_1 + \ell s_1
  \end{equation}
  where $d = \deg_{\ell} u_1 \in \mathbb{Z}_{\geq 0}$ is the degree of $u_1$ in its leader $\ell$.
\end{definition}

\begin{remark}
	In the Ritt pencil, both $t_1$ and $s_1$ are lower than $u_1$ in the variable $x_1$, by construction. 
\end{remark}

The following Proposition explains that the Ritt pencil satisfies the hypotheses of Theorem~\ref{T:fixed}.
\begin{proposition}\label{P:basic-containments}
	Let $\Sigma=V([u_1,\ldots,u_n])$. Let $\Sigma_1 \in \Irr(\Sigma)$ with generic point $\eta_1$. 
	Suppose we are in a degenerate situation for $u_1$ at $\eta_1$ and let $\Gamma$ be the corresponding Ritt pencil.
	Then 
	$$\Sigma_1 \subset \BS(\Gamma) \subset \Sigma.$$
\end{proposition}
\begin{proof}
	By definition, $[u_1,u_2,\ldots,u_n] \subset P_1$.
	Since we are in a degenerate situation, $s_1 \in P_1$.
	By definition (see Equation \eqref{E:t1-definition}), $t_1 = d\cdot u_1-\ell s_1$. 
	This implies that $\frak{b} \subset P_1$, i.e.~$\Sigma_1 \subset \BS(\Gamma)$. 
	
	Similarly, since $d\cdot u_1 = t_1 + \ell s_1$ we have that $[u_1,\ldots,u_n] \subset \frak{b}$, i.e.~$\BS(\Gamma) \subset \Sigma$.
\end{proof}

\begin{theorem}\label{T:moving-dimension}
	Let $\Sigma=V([u_1,\ldots,u_n])$. 
	Let $\Sigma_1 \in \Irr(\Sigma)$ with generic point $\eta_1$. 
	Suppose we are in a degenerate situation for $u_1$ at $\eta_1$ and let $\Gamma = V([t_1+ys_1,u_2,\ldots,u_n])$ be the Ritt pencil.
	
	Suppose the Dimension Conjecture. 
	If $\Gamma_1$ is a moving component of $\Gamma$, then $\Gamma_1$ has differential dimension 1. 
\end{theorem}
\begin{proof}
	Let $\Gamma_1 = V(Q)$ and let $\Sigma_1=V(P)$. 
	We now show $\trdeg_K^{\partial}(\kappa(Q))=1$. To do this, we take a lexicographic ordering from a block ranking which eliminates the variables $(x_1,x_2,\ldots,x_n)$ --- so it can be any ranking in $(x_1,x_2,\ldots,x_n)$ but $y$ needs to take priority over everything in the ranking. 
	Now consider a characteristic set for $Q$, which we write as 
	$$ (B_1,B_2,\ldots,B_s,A_1,A_2,\ldots, A_t),$$
	where $B_i \in K\lbrace x_1,x_2,\ldots,x_n\rbrace$ for $1\leq i \leq s$ and $A_j \in K\lbrace x_1,x_2,\ldots,x_n,y\rbrace \setminus K\lbrace x_1,x_2,\ldots,x_n\rbrace$ for $1\leq j \leq t$. 
	Note that we are forced to have $t\leq 1$ by elimination of $y$. 
	
	By the elimination theorem for characteristic sets (Proposition~\ref{P:characteristic-elimination}) we have that $(B_1,\ldots,B_s)$ is a characteristic set for $Q_1=Q \cap K\lbrace x_1,\ldots,x_n\rbrace$. 
	We note that $[u_2,u_3,\ldots,u_n] \subset Q_1$ and that $Q_1\subset P$.
	
	We now consider a chain of primes $0=P_n \subset P_{n-1} \subset \cdots \subset P_0 $ where $P_0=P$, $P_{1}$ is minimal over $[u_2,\ldots,u_n]$ which is contained in $P_0$, and for $i>0$ we proceed inductively saying that $P_{i}$ is a minimal prime over $[u_{i+1},\ldots,u_n]$ which is contained in $P_{i-1}$ etc.
	Here is a diagram:
	$$
	\begin{tikzcd}
		P_{n-1} \arrow[r, hook] & P_{n-2} \arrow[r, hook] & \cdots \arrow[r, hook] & P_1 \arrow[r, hook] & P_0=P  \\[4pt]
		{[u_n]} \arrow[u, hook] \arrow[r, hook]
		& {[u_{n-1},u_n]} \arrow[u, hook] \arrow[r, hook]
		& \cdots \arrow[r, hook]
		& {[u_2,\ldots,u_n]} \arrow[u, hook] \arrow[r, hook]
		& {[u_1,\ldots,u_n]} \arrow[u, hook].
	\end{tikzcd}
	$$
	We are precisely in the situation of Proposition~\ref{P:behavior-of-intermediates}, and hence $\trdeg^{\partial}(\kappa(P_i))=i$.

	We now have a situation where 	
	$$P \supset Q_1 \supset P_1\supset [u_2,u_3,\ldots,u_n], \text{ and }$$
	$$[u_2,u_3,\ldots,u_n] \subset [u_1,u_2,u_3, \ldots,u_n] \subset P.$$
	Since $Q_1 \supset P_1$,  $\trdeg^{\partial}(\kappa(Q_1)) \leq \trdeg^{\partial}(\kappa(P_1)) =1$.
  Since $(B_1,\ldots,B_s)$ is a characteristic set for $Q_1$, Proposition \ref{P:characteristic-sets-and-dimension} implies that $\trdeg^{\partial}(\kappa(Q_1)) = n-s$.
  In particular, $n-s \leq 1$, i.e., $n-1 \leq s$.

	We now analyze the various cases for $t$ and $s$ for the characteristic sets and derive a contradiction using the dimension conjecture.\footnote{This is an unfortunate collision in notation. In this case $s$ is the number of $B$'s in the characteristic set and $t$ is the number of $A$'s in the characteristic set. This is not to be confused with the separant and coseparant which are $s_1$ and $t_1$.}
	\begin{itemize}
		\item If $t=0$, then $Q$ has a characteristic set $[B_1,\ldots,B_s]$ with $B_i \in K\lbrace x_1,\ldots,x_n\rbrace$. Since $y s_1+t_1 \in Q$ we have that there exists some $s_0 \in S_{(B_1,\ldots,B_s)}$ (the multiplicative set generated by the initials and separants of the $B_i$) such that 
		$$ s_0 (y s_1+t_1) \equiv 0 \mod [B_1,\ldots, B_s].$$
		This implies that $$s_0s_1 \equiv 0 \mod [B_1,\ldots,B_s], \quad s t_1 \equiv 0 \mod [B_1,\ldots,B_s]$$ by linear independence of $K\lbrace x_1,\ldots,x_n\rbrace$ and $yK\lbrace x_1,\ldots,x_n\rbrace$.
		This implies that both $s_1$ and $t_1$ are contained in $Q$ since membership is characterized by the division algorithm. 
		This proves that $Q$ is contained in the base locus and hence does not correspond to a moving component (a contradiction). 
		
		\item If $t=1$ and $s=n-1$, then we are done as this would imply $\trdeg^{\partial}_K(\kappa(Q))=1$. 
		\item If $t=1$ and $s< n-1$, then $\trdeg^{\partial}(\kappa(Q))=0$, but we can drop in $\partial$-dimension by at most one after adding a new equation by the Dimension Conjecture. So this is also a contradiction. 
	\end{itemize}
\end{proof}

The following shows that how to use moving components of the Ritt pencil to get from a degenerate situation to a lower case of the Jacobi Bound Conjecture assuming the Dimension Conjecture.
\begin{theorem}\label{T:moving}
	Suppose $K$ is a differentially closed field of characteristic zero.
	Let $\Sigma=V([u_1,\ldots,u_n]) \subset \AA^n_{\infty}$. 
	Let $\Sigma_1 \in \Irr(\Sigma)$ with generic point $\eta_1$. 
	Suppose we are in a degenerate situation for $u_1$ at $\eta_1$ and let $\Gamma = V([t_1+ys_1,u_2,\ldots,u_n]) \to \AA^1_{\infty} = \Spec K\lbrace y \rbrace$ be the Ritt pencil.
	
	If $\Gamma_1 \in \Irr(\Gamma)$ is moving, then supposing the Dimension Conjecture there exists some $\mu\in K$ with $\Sigma_1 \subset \Gamma_{1,\mu} $ such that 
	\begin{enumerate}
		\item $\dim(\Gamma_{1,\mu}) <\infty$, and
		\item $J(s_1+\mu t_1,u_2,\ldots,u_n) \leq J(u_1,\ldots,u_n)$.
	\end{enumerate}

\end{theorem}

\begin{proof}
	Let $\Gamma_1=V(Q)$.
	Let $\Gamma_{1,\mu}=V(\widetilde{P})$ where $\widetilde{P}$ is the specialization of $Q$ at $\mu$.
	
	Let $Q \supset I_{\Gamma}$ be a minimal moving prime. 
	Assuming the Dimension Conjecture, Theorem~\ref{T:moving-dimension} implies  $\trdeg_K^{\partial}(\kappa(Q))=1$.
	Hence for each $i$ that there exists a non-trivial relation between $y$ and $x_i$ in $Q$. 
	More precisely there exists some $M_i(y,x_i) \in K\lbrace x_i,y\rbrace \cap Q $ which is not an element of $K\lbrace x_i \rbrace \cup K\lbrace y \rbrace$ (since it is a non-trivial $\partial$-algebraic dependence). 
	We then consider the ideal $[M_1(y,x_1),M_2(y,x_2),\ldots,M_n(y,x_n)] \subset Q $.
	By the differential algebraic version of the Nullstellensatz \cite[pg 33]{Buium1994} there exists some $\mu$ in our differentially algebraically closed field $K$ such that $M_1(\mu,x_1),M_2(\mu,x_2),\ldots,M_n(\mu,x_n)$ defines a non-trivial differential algebraic variety --- none of the equations are vanishing, none of the equations are constant. 
	We then get some $$\widetilde{P} \supset I_{\Gamma_{\mu}} = [t_1+\mu s_1, u_2, \ldots,u_n]$$
	minimal (over $I_{\Gamma_{\mu}}$) with inverse image $Q$ under the specialization map $\varphi_{\mu}\colon K\lbrace x_1,\ldots,x_n,y\rbrace \to K\lbrace x_1,\ldots,x_n \rbrace$ given by $\varphi_{\mu}(x_i) = x_i$ and $\varphi_{\mu}(y) = \mu$. Since $K\lbrace x_1,\ldots,x_n\rbrace/[M_1(\mu,x_1),\ldots,M_n(\mu,x_n)]$ has finite absolute dimension we conclude that $\widetilde{P}$ has finite absolute dimension.
	The reason this has finite absolute dimension is because $K\lbrace x_1,\ldots,x_n\rbrace/[M_1(\mu,x_1),\ldots,M_n(\mu,x_n)] \cong \bigotimes_{i=1}^n K\lbrace x_i \rbrace/[M_i(\mu,x_i)]$ and each $K\lbrace x_i \rbrace/[M_i(\mu,x_i)]$ has finite absolute dimension.
\end{proof}

\section{Dimension Conjecture Implies Jacobi Bound Conjecture}
\label{section:Dimension-Conjecture-Implies-Jacobi-Bound-Conjecture}

Recall from Definition \ref{D:ritt-ordering} Ritt's ordering on matrices.

\begin{theorem}\label{T:dc-implies-jbc}
	Assume the Dimension Conjecture (Conjecture~\ref{C:goosed}). 
	Then the Jacobi Bound Conjecture (Conjecture~\ref{C:jbc}) is true.
\end{theorem}
\begin{proof}
	Consider the system of equations $[u_1,u_2,\ldots,u_n] \subset K\lbrace x_1,x_2,\ldots,x_n \rbrace$. Let $P \supset [u_1,u_2,\ldots,u_n]$ be a minimal prime ideal such that $\trdeg(K\lbrace x_1,x_2,\ldots,x_n\rbrace/P)<\infty$. 
	
	The proof is by induction on $n$. Here $n$ is the number of variables, which is also equal to the number of equations.
	\\
	
 \noindent \emph{Base Case (Outer Induction)}: 
		The Jacobi Bound Conjecture is known in the case $n=1$ from the Ritt Bound (\cite[pg 135]{Ritt1950}).
		(The Jacobi Bound Conjecture in the case $n=2$ is also a theorem of Ritt \cite{Ritt1935}).
		\\
		
\noindent \emph{Inductive Step (Outer Induction)}: Assume the Jacobi Bound Conjecture for $m$ equations and $m$ variables with $m<n$. 
		The inductive step is another proof by induction, this time on the order matrices $A$. 
		\\
		
\noindent \emph{Base Case (Inner Induction)}: 
				These are order matrices $A$ of the form 
			$$ A = \begin{pmatrix}
				a_{1,1} & a_{1,2} & \cdots & a_{1,n} \\
				-\infty & a_{2,2} & \cdots & a_{2,n} \\
				-\infty & a_{3,2} & \cdots & a_{3,n} \\
				\vdots & \vdots & \ddots & \vdots \\
				-\infty & a_{n,2} & \cdots & a_{n,n}
			\end{pmatrix}. 
			$$
  Here we know that $\tdet(A) = a_{1,1} + \tdet(\widetilde{A}_{1,1})$, where $\widetilde{A}_{1,1}$ is the cofactor matrix for the entry $(1,1)$. The proof now reduces to the Jacobi Bound Conjecture in $(n-1)$-variables which we assume to be true by the outer inductive hypothesis.
  The proof is verbatim the same as the \emph{Base Case (Inner Induction)} in the proof of Corollary~\ref{C:linear-case}, where we proved the linear case. 
\\

\noindent \emph{Inductive Step (Inner Induction)}:
			We will show for all systems $[u_1,\ldots,u_n]$ with order matrix $A$ and minimal prime ideals $P \supset [u_1,\ldots,u_n]$ with $\trdeg_K^{\partial}(\kappa(P))<\infty$, there exists a system $[v_1,\ldots,v_n]$ with order matrix $B\prec A$ and minimal prime $\widetilde{P}\supset [v_1,\ldots,v_n]$ with $\trdeg_K^{\partial}(\kappa(\widetilde{P}))<\infty$ such that 
			$$P \supset \widetilde{P} \supset [v_1,\ldots,v_n], \quad J(v_1,\ldots,v_n) \leq J(u_1,\ldots,u_n).$$
Once we show this,  then by the inductive hypothesis the Jacobi Bound Conjecture will hold for the system $[v_1,\ldots,v_n]$ with the prime ideal $\widetilde{P}$, and observing that
			$$\dim V(P)\leq \dim V(\widetilde{P}) \leq J(v_1,\ldots,v_n) \leq J(u_1,\ldots,u_n)$$
  proves that JBC also holds for $P$
\\  
			
\noindent \emph{Proof of Main Claim of Inner Induction:}	We now show the main claim of the inner inductive step. 
	Adapting Ritt's technique we will modify our order matrix to be in one of two standard forms where we can then reduce the first column via some Ritt division. One of our contributions is the observation that Ritt's reductions in these steps still hold provided that separants aren't vanishing on the component in question (Proposition~\ref{P:ritt-first} and Proposition~\ref{P:ritt-second}). 
\\
	
\noindent \emph{Proof of Main Claim for Ritt's First Form (non-degenerate)}: Suppose that the order matrix $A$ has $\tdet(A) = \sum_{i=1}^n a_{i,i}$. If $a_{1,1}$ is not the unique maximum of the first column and $\partial u_1/\partial \ell_{x_1} \not \in P$ (where $\ell_{x_1} = x_1^{(a_{1,1})}$), is non-vanishing, then we apply Ritt's first reduction (Proposition \ref{P:ritt-first}) and get am equivalent system $[v_1,\ldots,v_n]$ with order matrix $B\prec A$ and minimal prime ideal $P=\widetilde{P} \supset [v_1,\ldots,v_n]$.
\\

\noindent \emph{Proof of Main Claim for Ritt's Second Form (non-degenerate)}: If $a_{1,1}$ is the unique maximum, then we can rearrange so that the second largest element appears in the $(1,1)$-entry with $a_{1,1}$ now moving to the $(n,1)$-position. 
		If $\partial u_1/\partial \ell_x\notin P$ we can apply Ritt's second reduction (Proposition~\ref{P:ritt-second}) and get an equivalent system $[v_1,v_2,\ldots,v_n]$ order matrix $B\prec A$ with $P=\widetilde{P}\supset [v_1,v_2,\ldots,v_n]$. 
\\

\noindent \emph{Proof of Main Claim for Degenerate Situation (Overview)}:	Let $\Sigma_1 = V(P)$ and let $\eta_1$ be its generic point. 
		If we are in neither of the previous cases we are in a degenerate situation (Definition~\ref{D:degenerate})  and 
		 $$s_1=\partial u_1/\partial \ell_x \in P.$$
		Equivalently $s_1$ vanishes on $\eta_1$.
		Let $\Gamma \to \AA^1_{\infty}=\Spec(K\lbrace y \rbrace)$ be the Ritt pencil associated to this degenerate situation so that $\Gamma = \Spec(K\lbrace x_1,\ldots,x_n\rbrace/[t_1+ys_1,u_2,\ldots,u_n]$.
		
		By Proposition~\ref{P:basic-containments} we have 
		$$ \Sigma_1 \subset \BS(\Gamma) \subset \Sigma. $$
		Let $\Gamma_1$ be an irreducible component of $\Gamma$ containing $\Sigma_1$. 
		We will let $Q \subset K\lbrace x_1,\ldots,x_n,y\rbrace$ be the prime differential ideal such that $\Gamma_1 = V(Q)$. 
		
		Similarly, since $\Sigma_1 \subset \BS(\Gamma) \subset \Gamma_{\mu}$ for every $\mu$ we can let $\Gamma_{\mu,1}$ be an irreducible component of $\Gamma_{\mu}$ containing $\Gamma_1$. 
		We will let $\widetilde{P} \subset K\lbrace x_1,\ldots,x_n\rbrace$ be the prime differential ideal such that $\Gamma_{\mu,1} = V(\widetilde{P})$ and we may suppose that $\Gamma_{\mu,1} \subset \Gamma_1$ without loss of generality. 
		
		We can make the same reasoning algebraically: $P \owns u_1,s_1$ and hence $P \owns t_1$. 
		Since $\widetilde{P}$ is a minimal prime over $[t_1+\mu s_1,u_2,\ldots,u_n]$ we see that $[t_1+\mu s_1,u_2,\ldots,u_n] \subset P$ and hence can take $\widetilde{P} \subset P$. (We take $Q \subset P K\lbrace x_1,\ldots,x_n,y \rbrace$ by similar reasoning.)
\\
		
\noindent \emph{Proof of Main Claim for Degenerate Situation (Fixed Case)}:
		Suppose that $\Gamma_{1}$ is a fixed component. 
			By theorem~\ref{T:fixed}, any component of $\Gamma_1 \in \Irr(\Gamma)$ which is fixed is equal to $\Sigma_1$. 
			This means that $\Gamma_{\mu,1}$ is fixed and equal to $\Sigma_1$ with $\widetilde{P}=P\supset [t_1+\mu s_1,u_2,\ldots,u_n]$ with order matrix $B\prec A$.
			\\
			
\noindent \emph{Proof of Main Claim for Degenerate Situation (Moving Case)}: 
Suppose that $\Gamma_1$ is a moving component.
			In Theorem~\ref{T:moving-dimension} we proved that $\dim^{\partial}(\Gamma_1)=1$, assuming the Dimension Conjecture. 
			In Theorem~\ref{T:moving}, (assuming Theorem~\ref{T:moving-dimension}) we showed that there exists some $\mu$ such that the any moving component of  $\Gamma_{\mu}$ contained in $\Gamma_1 \subset \Gamma \subset \AA^n_{\infty}\times \AA^1_{\infty}$ has finite absolute dimension with Jacobi number no greater than the Jacobi number of our original system.
			This is the part of the proof that gives a true deformation.
			There is some component $\Gamma_1$ with $\Gamma_{\mu,1} \supset \Sigma_1$, where $\dim \Gamma_{\mu,1}<\infty$, and equations given by $[t_1+\mu s_1,u_2,\ldots,u_n]$ with order matrix $B\prec A$.
\end{proof}

\providecommand{\bysame}{\leavevmode\hbox to3em{\hrulefill}\thinspace}
\providecommand{\MR}{\relax\ifhmode\unskip\space\fi MR }
\providecommand{\MRhref}[2]{%
  \href{http://www.ams.org/mathscinet-getitem?mr=#1}{#2}
}
\providecommand{\href}[2]{#2}

\end{document}